\documentclass{amsart}

\usepackage{amsmath,amssymb,exscale}
\usepackage{ams}
\usepackage{psfrag}
\usepackage{graphicx, color}
\usepackage{enumerate}
\usepackage{tikz}
\usetikzlibrary{calc,fadings,decorations.pathreplacing,arrows,
datavisualization.formats.functions,shapes.geometric}
\usepackage{verbatim}

\newcommand{\corb}[1]{\textcolor{black}{#1}}

\newcommand{\mcA}{{\mathcal A}}
\newcommand{\mcB}{{\mathcal B}}
\newcommand{\mcC}{{\mathcal C}}
\newcommand{\mcD}{{\mathcal D}}
\newcommand{\mcE}{{\mathcal E}}
\newcommand{\mcF}{\mathcal{F}}
\newcommand{\mcG}{\mathcal{G}}
\newcommand{\mcI}{{\mathcal I}}

\newcommand{\mcH}{{\mathcal H}}

\newcommand{\mcO}{{\mathcal O}}
\newcommand{\mcP}{{\mathcal P}}
\newcommand{\mcQ}{{\mathcal Q}}
\newcommand{\mcR}{{\mathcal R}}
\newcommand{\mcS}{\mathcal{S}}

\newcommand{\Real}{\mathop{\text{\rm Re}}}
\newcommand{\Imag}{\mathop{\text{\rm Im}}}

\newcommand{\jj}{\mathbf{j}}
\newcommand{\pp}{\mathbf{p}}

\newcommand{\bw}{{\bf w}}
\newcommand{\by}{{\bf y}}
\newcommand{\bz}{{\bf z}}
\newcommand{\bY}{\mathbf{Y}}

\newcommand{\ii}{\mathbf{i}}

\newcommand{\mm}{\mathbf{m}}

\newcommand{\bys}{{\bf y}_s}
\newcommand{\byf}{{\bf y}_f}
\newcommand{\byo}{{\bf y}_{0}}
\newcommand{\byd}{\delta {\bf y}}
\newcommand{\bo}{{\bf 0}}

\newcommand{\oone}{\boldsymbol{1}}

\newcommand{\gset}{{\mathbb G}}

\newcommand{\Nset}{{\mathbb N}}

\newcommand{\sset}{{\mathbb S}}
\newcommand{\Pol}{\mathbb{P}}
\newcommand{\eset}[1]{{\mathbb E} \left[ #1 \right] }
\newcommand{\esets}[1]{{\mathbb E} [ #1 ] }

\newcommand{\Grad}{\nabla}

\newcommand{\var}{\text{\rm var}}

\newcommand{\lv}{w}

\newtheorem{assumption}{Assumption}{\bf}{\rm}
\newtheorem{lemma}{Lemma}
\newtheorem{theorem}{Theorem}

\newtheorem{remark}{Remark}
\newtheorem{problem}{Problem}

\newtheorem{definition}{Definition}

\newcommand{\tu}{(\tilde{u} \circ F)(\cdot,\by)}
\newcommand{\tus}{(\tilde{u} \circ F)(\cdot,\by + s \delta \by)}
\newcommand{\tf}{(f \circ F)(\cdot,\by)}
\newcommand{\tw}{{\bf w}(\by)}  

\newcommand{\0}{\mathbf{0}}

\DeclareMathOperator*{\esssup}{ess\,sup}
\DeclareMathOperator*{\essinf}{ess\,inf}

\begin{document}

\title[Hybrid collocation perturbation for PDEs with random
  domains]{A hybrid collocation-perturbation approach for PDEs with
  random domains}

\author{Julio E. Castrill\'{o}n-Cand\'{a}s
  \and Fabio Nobile
  \and Ra\'{u}l F. Tempone}

\begin{abstract}
In this work we consider the problem of approximating the statistics
of a given Quantity of Interest (QoI) that depends on the solution of
a linear elliptic PDE defined over a random domain parameterized by
$N$ random variables. \corb{The random domain is split into large and
  small variations contributions.  The large variations are
  approximated by applying a sparse grid stochastic collocation
  method. The small variations are approximated with a stochastic
  collocation-perturbation method.}  Convergence rates for the
variance of the QoI are derived and compared to those obtained in
numerical experiments. Our approach significantly reduces the
dimensionality of the stochastic problem. \corb{The computational cost
  of this method increases at most quadratically with respect to the
  number of dimensions of the small variations. Moreover, for the case
  that the small and large variations are independent the cost
  increases linearly.}
\end{abstract}

\maketitle

\section{Introduction}

The problem of design under the uncertainty of the underlying domain
can be encountered in many real life applications. For example, in
semiconductor fabrication the underlying geometry becomes increasingly
uncertain as the physicals scales are reduced \cite{Zhenhai2005}. This
uncertainty is propagated to an important Quantity of Interest (QoI)
of the semiconductor circuit. If the variance of the capacitance is
high this could lead to low yields during the manufacturing
process. It is important to quantify the uncertainty of the QoI in the
circuit to be able to maximize yields. This will have a direct impact
in reducing the costly and time-consuming design cycle. Other examples
included graphene nano-sheet fabrication \cite{Hout:2010}. In this
paper we focus on the problem of how to efficiently compute the
statistics of the QoI given uncertainty in the underlying geometry.

Uncertainty Quantification (UQ) methods applied to Partial
Differential Equations (PDEs) with random geometries can be mostly
divided into collocation and perturbation approaches. For large
deviations of the geometry the collocation method
\cite{Chauviere2006,Fransos2008,Tartakovsky2006,Castrillon2013} is
well suited. In addition, in \cite{Castrillon2013,Harbrecht2014} the
authors derive error estimates of the solution with respect to the
number of stochastic variables in the geometry description. However,
this approach is effective for a moderate number of stochastic
variables.  In contrast, the perturbation approaches introduced in
\cite{Harbrecht2008,Zhenhai2005} are efficient for high dimensional
small perturbations of the domain.

\corb{We represent the domain in terms of a series of random variables
  and then remap the corresponding PDE to a deterministic domain with
  random coefficients.} The random geometry is split into small and
large deviations.  A collocation sparse grid method is used to
approximate the contribution to the QoI from the first large
deviations $N_s$ terms of the stochastic domain expansion. Conversely,
the contribution of the small deviations (the tail) are cheaply
computed with a collocation and perturbation method.

We derive rigorous convergence analysis of the statistics of the QoI
in terms of the number of collocation points and the perturbation
approximation of the tail. \corb{Analytic estimates show that the
  error of the QoI for the hybrid collocation-perturbation method (or
  the hybrid perturbation method for short) decays quadratically with
  respect to the of sum of the series coefficients of the series
  expansion of the tail. This is in contrast to the linear decay of
  the error estimates derived in \cite{Castrillon2013} for the pure
  stochastic collocation approach.} Furthermore, numerical experiments
show faster convergence than the stochastic collocation approach.

The outline of the paper is the following: In Section \ref{background}
mathematical background material is introduced.  In Section
\ref{setup} we set up the problem and reformulate the random domain
elliptic PDE problem onto a deterministic domain with random matrix
coefficients. We assume that the random boundary is parameterized by
$N$ random variables. In section \ref{collocation-perturbation} we
derive the hybrid collocation-perturbation approach. \corb{The
  approach reduces to computing {\it mean} and {\it variance
    correction} terms that quantifies the contribution from the tail
  of the random domain expansion.  In Section \ref{analyticity} we
  show that the mean and variance correction terms can be analytically
  extended onto a well defined region in $\C^{N_s}$.}  In Section
\ref{erroranalysis} we derive error estimates for the mean and
variance of the QoI with respect to the finite element, sparse grid
and perturbation approximations. In section \ref{complexity} a
complexity and tolerance analysis is derived.  Finally, in section
\ref{numericalresults} numerical examples are presented.

\section{Background}
\label{background}

In this section we introduce the general notation and mathematical
background that will be used in this paper. Let $\Omega$ be the set of
outcomes from the complete probability space $(\Omega, \mcF,
\mathbb{P})$, where $\mcF$ is a sigma algebra of events and
$\mathbb{P}$ is a probability measure.  \corb{Define
  $L^{q}_{\Pol}(\Omega)$, $q \in [1, \infty]$, as the following Banach
  spaces:}
\[
\corb{L^{q}_{\Pol}(\Omega) := \{v \,\,|\,\, \int_{\Omega}
|v(\omega)|^{q}\,d\Pol(\omega) < \infty \}\,\mbox{and}\,
L^{\infty}_{\Pol}(\Omega) := \{v \,\,|\,\, \esssup_{ \omega \in \Omega}
|v(\omega)| < \infty \},}
\]
\noindent where $v:\Omega \rightarrow \mathbb{R}$ is a measurable
random variable.

Let $\bY:=[Y_1, \dots, Y_{N}]$ be a $N$ valued random vector
measurable in $(\Omega, \mcF, \mathbb{P})$ and without loss of
generality denote $\Gamma_n:=[-1,1]$ as the image of $Y_n$ for
$n=1,\dots,N$. Assume that $\bY$ takes values on $\Gamma:=\Gamma_{1}
\times \dots \times \Gamma_{N} \subset \mathbb{R}^{N}$ and let
$\mcB(\Gamma)$ be the Borel $\sigma-$ algebra.  Define the induced
measure $\mu_{\bY}$ on $(\Gamma,\mcB(\Gamma))$ as $\mu_{\bY} : =
\mathbb{P}(\bY^{-1}(A))$ for all $A \in \mcB(\Gamma)$. Assuming that
the induced measure is absolutely continuous with respect to the
Lebesgue measure defined on $\Gamma$, then there exists a density
function $\rho({\bf y}): \Gamma \rightarrow [0, +\infty)$ such that
  for any event $A \in \mcB(\Gamma)$
\[
\mathbb{P}(\bY \in A) := \mathbb{P}(\bY^{-1}(A)) = \int_{A} \rho( {\bf y}
)\,d {\bf y}.
\]
\noindent Now, for any measurable function $\bY \in L^{1}_{P}(\Gamma)$
\corb{define the expected value as}
\[
\mathbb{E}[\bY] = \int_{\Gamma} {\bf y} \rho( {\bf y} )\, d
{\bf y}.
\]
Define also the following Banach spaces:
\[
\begin{split}
L^{q}_{\rho}(\Gamma) 
&:= \left\{ v \,\,|\,\, \int_{\Omega}
|v(\by)|^{q}\, \rho(\by) d\by < \infty \right\} \,\,\mbox{and}\,\,\, \\
L^{\infty}_{\rho}(\Gamma) 
&:= \left\{v \,\,|\,\, \esssup_{ \by \in \Gamma}
|v(\by)| < \infty \right\}.
\end{split}
\]

\corb{We discuss in the next section an approach of approximating a
  given function $\tilde f \in L^{q}_{\rho}(\Gamma)$, which is
  sufficiently smooth, by multivariate polynomials and sparse grid
  interpolation.}

\subsection{Sparse Grids}

Our goal is to find a compact an accurate approximation of a
multivariate function $\tilde f:\Gamma \rightarrow \R$ with sufficient
regularity. \corb{It is assumed that $\tilde f \in C^{0}(\Gamma; V)$ where
\[
C^{0}(\Gamma;V) := \{\mbox{$v:\Gamma \rightarrow V$ is continuous on
  $\Gamma$ and $\max_{\by \in \Gamma} \|v(\by)\|_{V} < \infty$ } \}
\]
and $V$ is a Banach space.}  Consider the univariate Lagrange
interpolant along the $n^{th}$ dimension of $\Gamma$
\[
\mcI_n^{m(i)}:C^0(\Gamma_{n})\rightarrow \mcP_{m(i)-1}(\Gamma_{n}),
\]
where $i\geq 1$ denotes the level of approximation and $m(i)$ the
number of collocation knots used to build the interpolation at level
$i$ such that $m(0)=0$, $m(1)=1$ and $m(i) < m(i+1)$ for $i\geq
1$. Furthermore let $\mcI_n^{m(0)}=0$. The space
$\mcP_{m(i)-1}(\Gamma_{n})$ is the set of polynomials of degree at
most $m(i) - 1$.

We can construct an interpolant by taking tensor products of
$\mcI_n^{m(i)}$ along each dimension for $n = 1,\dots,N$. However, the
number of collocation knots explodes exponentially with respect to the
number of dimensions, thus limiting feasibility to small dimensions.
Alternately, consider the difference operator along the $n^{th}$
dimension
\[
  \Delta_n^{m(i)} := \mcI_n^{m(i)}-\mcI_n^{m(i-1)}.
\]
\corb{The sparse grid approximation of $\tilde f \in C^{0}(\Gamma)$ is
  defined as}
\begin{equation}
  \mcS^{m,g}_w[\tilde f] 
= \sum_{\ii\in\Nset^{N}_+: g(\ii)\leq w} \;\;
 \bigotimes_{n=1}^{N} \Delta_n^{m(i_n)}(\tilde f) 
\end{equation}
where $w \geq 0$, $w \in \N_+$ $(\N_+ := \N \cup \{0\})$, is the
approximation level, $\ii=(i_1,\ldots,i_{N})$ $\in \Nset^{N}_+$, and
$g:\Nset^{N}_+\rightarrow\Nset$ is strictly increasing in each
argument. The sparse grid can also we re-written as
\begin{equation}
  \mcS^{m,g}_w[\tilde f]
  = \sum_{\ii\in\Nset^{N}_+: g(\ii)\leq w} \;c(\ii)\; 
  \bigotimes_{n=1}^{N} \mcI_n^{m(i_n)}(\tilde f), \qquad \text{with } c(\ii) 
  = \sum_{\stackrel{\jj \in \{0,1\}^{N}:}{g(\ii+\jj)\leq w}} (-1)^{|\jj|}.
\label{sparsegrid:eqn1}
\end{equation}
From the previous expression, we see that the sparse grid
approximation is obtained as a linear combination of full tensor
product interpolations.  \corb{However, the constraint $g(\ii)\leq w$
  in \eqref{sparsegrid:eqn1} restricts the growth of tensor grids of
  high degree.}

\corb{Let $\mm(\ii) = (m(i_1),\ldots,m(i_{N}))$ and consider the ordered
polynomial polynomial set}
\[
\Lambda^{m,g}(w) = \{\pp\in\Nset^{N}, \;\;  g(\mm^{-1}(\pp+\oone))\leq w\}. 
\]
\corb{Let $\Pol_{\Lambda^{m,g}(w)}(\Gamma)$ be the associated multivariate
polynomial space}
\[
\Pol_{\Lambda^{m,g}(w)}(\Gamma) = span\left\{\prod_{n=1}^{N} y_n^{p_n},
  \;\; \text{with } \pp\in\Lambda^{m,g}(w)\right\}.
\]
\corb{It can shown that $\mcS^{m,g}_w[\tilde f] \in
  \Pol_{\Lambda^{m,g}(w)}(\Gamma)$ (see e.g. \cite{Back2011}). Now,
  one of the most typical choices for $m$ and $g$ is given by the
  Smolyak (SM) formulas (see
  \cite{Smolyak63,Novak_Ritter_00,Back2011})}
\[
m(i) = \begin{cases} 1, & \text{for } i=1 \\ 2^{i-1}+1, & \text{for }
  i>1 \end{cases}\quad \text{ and } \quad g(\ii) = \sum_{n=1}^N
(i_n-1).
\]
This choice of $m$, combined with the choice of Clenshaw-Curtis (CC)
interpolation points (extrema of Chebyshev polynomials) leads to
nested sequences of one dimensional interpolation formulas and a
sparse grid with a highly reduced number of points compared to the
corresponding tensor grid (see \cite{Back2011}). Other choices are
given by Total Degree (TD) and Hyperbolic Cross (HC).

\corb{It can also be shown that the TD, SM and HC anisotropic sparse
  approximation formulas can be readily constructed with improved
  convergence rates (see \cite{nobile2008b}). Moreover, in
  \cite{Chkifa2014}, the authors show convergence of anisotropic
  sparse grid approximations with infinite dimensions ($N \rightarrow
  \infty$).}

\corb{In \cite{Nobile2016} the authors show the construction of
  quasi-optimal grids have been shown to have exponential
  convergence.}

\section{Problem setup and formulation}
\label{setup}

\corb{Let $D(\omega) \subset \mathbb{R}^{d}$ be an open bounded domain
  with Lipschitz boundary $\partial D(\omega)$ that is shape dependent
  on the stochastic parameter $\omega \in \Omega$ and and a Lipschitz
  bounded open reference domain $U \subset \R^d$.} Let the map
$F(\omega):U \rightarrow D(\omega)$ be a one-to-one for all $\omega
\in \Omega$ and whose image coincides with
$D(\omega)$. \corb{Furthermore denote $\partial F(\omega)$ as the
  Jacobian of $F(\omega)$ and suppose that $F$ satisfies the following
  assumption.}

\begin{assumption} Given a one-to-one map $F(\omega):U \rightarrow
  D(\omega)$ there exist constants $\F_{min}$ and $\F_{max}$ such that
\[
0<\F_{min} \leq \sigma_{min} (\partial F(\omega)) 
\,\,\mbox{and}\,\, \sigma_{max} (\partial F(\omega)) \leq
\F_{max} < \infty
\]
\label{setup:Assumption1}
\noindent almost everywhere in $U$ and almost surely in $\Omega$. We
have denoted by $\sigma_{min} (\partial F(\omega))$ (and $\sigma_{max}
(\partial F(\omega))$) the minimum (respectively maximum) singular
value of the Jacobian $\partial F(\omega)$. \corb{In Figure
  \ref{setup:fig1} a cartoon example of the deformation of the
  reference domain $U$ is shown.}
\end{assumption}

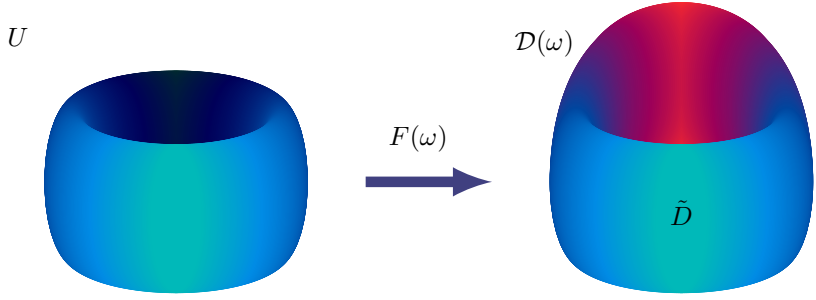
\begin{figure}
\begin{center}
\begin{tikzpicture}[scale=1.4] 
    \foreach \x in {90,...,-90} { 
    \pgfmathsetmacro\elrad{10*max(sin(\x),.7)}
    \pgfmathsetmacro\ltint{.9*abs(\x-55)/180}
    \pgfmathsetmacro\rtint{.9*(1-abs(\x+55)/180)}
    \definecolor{currentcolor}{rgb}{0, \ltint, \rtint}
    \draw[color=currentcolor,fill=currentcolor] 
        (xyz polar cs:angle=\x,y radius=0.35,x radius=1.0) 
        ellipse (\elrad pt and 20pt);
    \definecolor{currentcolor}{rgb}{0, \ltint, \rtint}
    \draw[color=currentcolor,fill=currentcolor] 
        (xyz polar cs:angle=180-\x,radius=0.35,x radius=1.0) 
        ellipse (\elrad pt and 20pt);
    } 
\coordinate (O) at (1.8,0);
\coordinate (P) at (3.0,0);
\draw[->, >=latex, green!50!red!50!blue, line width=4 pt] (O) -- (P);
\node[above=16pt,scale=1] at (2.3,-0.2) {$F(\omega)$};
\node[above=16pt,scale=1] at (-1.5,0.8) {$U$};
\end{tikzpicture}
\begin{tikzpicture}[scale=1.4] 
    \foreach \x in {90,...,0} { 
    \pgfmathsetmacro\elrad{10*max(sin(\x),.7)}
    \pgfmathsetmacro\ltint{.9*abs(\x-55)/180}
    \pgfmathsetmacro\rtint{.9*(1-abs(\x+55)/180)}
    \pgfmathsetmacro\ztint{2*(abs(\x)/180)}
    \definecolor{currentcolor}{rgb}{\ztint,\ltint,\rtint}
    \draw[color=currentcolor,fill=currentcolor] 
        (xyz polar cs:angle=\x,y radius=1,x radius=1.0) 
        ellipse (\elrad pt and 20pt);
    \definecolor{currentcolor}{rgb}{\ztint,\ltint,\rtint}
    \draw[color=currentcolor,fill=currentcolor] 
        (xyz polar cs:angle=180-\x,radius=1,x radius=1.0) 
        ellipse (\elrad pt and 20pt);
    }

    \foreach \x in {0,...,-90} { 
    \pgfmathsetmacro\elrad{10*max(sin(\x),.7)}
    \pgfmathsetmacro\ltint{.9*abs(\x-55)/180}
    \pgfmathsetmacro\rtint{.9*(1-abs(\x+55)/180)}
    \definecolor{currentcolor}{rgb}{0, \ltint, \rtint}
    \draw[color=currentcolor,fill=currentcolor] 
        (xyz polar cs:angle=\x,y radius=.35,x radius=1.0) 
        ellipse (\elrad pt and 20pt);
    \definecolor{currentcolor}{rgb}{0, \ltint, \rtint}
    \draw[color=currentcolor,fill=currentcolor] 
        (xyz polar cs:angle=180-\x,radius=.35,x radius=1.0) 
        ellipse (\elrad pt and 20pt);
    }
    \pgfmathsetmacro\elrad{cos(-135)}
    \pgfmathsetmacro\xrad{1.5cm-20pt*\elrad}
    \pgfmathsetmacro\yrad{.75cm-20pt*sin(-135)}

    \node[above=16pt,scale=1] at (0,-0.9) {$\tilde{D}$};
    \node[above=16pt,scale=1] at (-1.3,0.7) {$\mcD(\omega)$};
\end{tikzpicture}
\end{center}
\caption{\corb{Cartoon example of stochastic domain realization from a
  reference domain. The front of the torus, shown by the area $\tilde
  D$ is not stochastic and thus not deformed. The back of the torus is
  deformed from the reference domain $U$.  This figure is modified
  from the TikZ tex code from {\it Smooth map of manifolds and smooth
    spaces} by Andrew Stacey \cite{Stacey}.}}
\label{setup:fig1}
\end{figure}


\begin{lemma}
\corb{Under Assumptions \ref{setup:Assumption1} it is immediate to
  prove the following results:}
\begin{enumerate}[i)]
\item $L^{2}(D(\omega))$ and $L^{2}(U)$ are isomorphic almost surely.
\item $H^{1}(D(\omega))$ and $H^{1}(U)$ are isomorphic almost surely.
\end{enumerate}
\label{setup:lemma0}
\end{lemma}
\begin{proof} see \cite{Castrillon2013}.
\end{proof}

Now, consider the following boundary value problem: Given
$f(\cdot,\omega), a(\cdot,\omega): D(\omega) \rightarrow \R^{d}$ and
$g(\cdot,\omega):\partial D(\omega) \rightarrow \R^{d}$ find
$u(\cdot,\omega): D(\omega) \rightarrow \R^{d}$ such that almost
surely
\begin{equation}
\corb{
\begin{split}
  -\nabla \cdot ( a(x, \omega) \nabla u(x,\omega) ) &= f(x,\omega),\,\,\,
x \in D(\omega) \\
u &= g\hspace{11mm}\mbox{on $\partial D(\omega)$.}
\end{split}
}
\label{setup:eqn1}
\end{equation}
We now make the following assumption:
\begin{assumption}There exist constants $a_{min}$ and $a_{max}$ such
  that
\[
0 < a_{min} \leq a(x, \omega) \leq a_{max} < \infty \,\,\,\mbox{for
  a.e. $x \in D(\omega)$, and a.s. $\omega \in \Omega$}.
\]
\noindent where
\[
a_{min} := \essinf_{x \in D(\omega), \omega \in \Omega } a(x, \omega)
\,\,\,\,\mbox{and}\,\,\,\, a_{max} := \esssup_{x \in D(\omega),
  \omega \in \Omega} a(x, \omega).
\]
\label{setup:Assumption2}
\end{assumption}

Since $U$ is bounded and Lipschitz there exists a bounded linear
operator $T:H^{1/2}(\partial U) \rightarrow H^{1}(U)$ such that for
all $\tilde g \in H^{1/2}(\partial U)$ we have that $\tilde {\bf w} :=
T \tilde g \in H^{1}(U)$ satisfies $\hat \bw|_{U} = \tilde g$ almost
surely. \corb{By applying a change of variables the weak form of
  \eqref{setup:eqn1} can be reformulated on the reference domain $U$
  (see \cite{Castrillon2013} for details) as:}
\begin{problem}
Given that $(f \circ F)(\cdot,\omega) \in L^{2}(U)$ find
$\hat{u}(\cdot, \omega) \in H^{1}_{0}(U)$ s.t.
\begin{equation}
  B(\omega ;\hat{u},v) = 
\tilde{l}(\omega; v) ,\,\,\,\forall v \in H^{1}_{0}(U)
\label{setup:problem1}
\end{equation}
\corb{almost surely, where $\tilde{l}(\omega; v):=\int_{U} (f \circ
  F)(\cdot,\omega) | \partial F(\cdot,\omega) |v - L(\hat
  \bw(\cdot,\omega),v)$, $\hat g := g \circ F$, $\hat \bw := T(\hat
  g)$, for any $w,s \in H^1_{0}(U)$}
\[
\begin{split}
\corb{
B(\omega; s,w)}
&:= 
\corb{
\int_{U} (a \circ F(\cdot,\omega) \nabla s
^{T} C^{-1}(\cdot,\omega) \nabla w | \partial
F(\cdot,\omega)|,}
\\
L( \hat \bw(\cdot,\omega) ,v)
&:=
   \int_U (a \circ F)(\cdot,\omega) (\nabla (\hat \bw
(\cdot,\omega))^{T} C^{-1}(\cdot,\omega) | \partial F(\cdot,
  \omega) | \nabla v, 
\end{split}
\]
$C(\cdot,\omega) := \partial F(\omega) ^T \partial F(\omega)$, and
$\hat \bw(\cdot,\omega) |_{\partial U} = \hat g(\cdot,\omega)$.  This
homogeneous boundary value problem can be remapped to $D(\omega)$ as
$\tilde{u}(\cdot,\omega) := (\hat{u} \circ F^{-1})(\cdot, \omega)$,
thus we can rewrite $\hat{u}(\cdot,\omega) = (\tilde{u} \circ
F)(\cdot,\omega)$.
\label{setup:Prob3}
\end{problem}
The solution $u(\cdot, \omega) \in H^{1}(D(\omega))$ for the Dirichlet
boundary value problem is obtained as $u(\cdot,\omega) = \tilde{u}
(\cdot,\omega) + (\hat \bw \circ F^{-1})(\cdot,\omega)$.

\subsection{Quantity of interest and the adjoint problem}
\label{setup:QoI}

\corb{For many practical problems the QoI is not necessarily the
  solution of the elliptic PDE, but instead a bounded linear
  functional $Q:H^{1}_{0}(U) \rightarrow \R$ of the solution. This
  could be for example the average of the solution on a specific
  region of the domain.  Let us consider the QoI of the form}
\begin{equation}
\corb{Q(u) := \int_{\tilde{D}} q(x) u(x,\omega)\,\,dx}
\label{setup:qoi}
\end{equation}
with $q \in L^{2}(\tilde D)$ over the region $\tilde{D} \subset
D(\omega)$ for any $\omega \in \Omega$.  \corb{It is assumed that
  there $\exists \delta > 0$ such that $dist(\tilde{D}, \partial
  D(\omega)) > \delta$ for all $\omega \in \Omega$ and $F
  \left. \right|_{\tilde{D}} = I$ on $\tilde{D}$. In layman's terms
  the region $\tilde D$ has no deformations and it is contained inside
  $D$. This, for example, could be a small patch inside $D$ that is
  known not to be deformed.}

\corb{
\begin{remark}
  The restriction $F \left. \right|_{\tilde{D}} = I$ on $\tilde D$ is
  not hard. This is done to simplify the numerical simulations in
  Section \ref{numericalresults}.  The perturbation approach in
  Section \ref{collocation-perturbation} and the analyticity analysis
  in Section \ref{analyticity} are still valid even if this
  restriction is relaxed.
  \end{remark}
}

\corb{In the next section, the perturbation approximation is derived
  for $Q(u)$ and not directly from the solution $u$. It is thus
  necessary to introduce the influence function $\varphi:H^{1}_{0}(U)
  \rightarrow \R$, that can be easily computed by the following
  adjoint problem:
\begin{problem}
Find $\varphi \in H^{1}_{0}(U)$ such that for all $v \in H^{1}_{0}(U)$
\begin{equation}
B(\omega; v, \varphi) = Q(v)
\label{setup:dual2}
\end{equation}
a.s. in $\Omega$. After computing the influence function $\varphi$,
the QoI can be computed as $Q(u) = B(u, \varphi)$.
\label{setup:Prob4}
\end{problem}
  \begin{remark}
    We can pick a particular operator $T$ such that $\hat \bw = T(\hat
    g)$ and vanishes in the region defined by $\tilde D$. Thus we have
    that $Q(\hat \bw) = 0$ and $Q(u) = Q(\tilde u + \hat \bw) =
    Q(\tilde u)$.
    \end{remark}
}

\subsection{Domain parameterization and semi-discrete approximation}
\label{setup:domainparametrization}

To simplify the analysis of the elliptic PDE with a random domain from
equation \eqref{setup:eqn1} we remapped the solution onto a fix
deterministic reference domain. This approach has also been applied in
\cite{Fransos2008,Castrillon2013}. We now restrict our attention to
a particular class of domain deformation.

\begin{assumption} 
The map $F(\omega):U \rightarrow D(\omega)$ has the form
\[
F(x, \omega) := x + e(x,\omega)\hat{v}(x)
\]
a.s. in $\Omega$, with $\hat{v}:U \rightarrow \R^{d}$, $\hat v :=
[\hat v_1, \dots, \hat v_d]^T$, $\hat{v}_{i} \in C^{1}(U)$ for $i = 1,
\dots, d$, and $e(\cdot,\omega): U \rightarrow
D(\omega)$. Assume that the map $F(\omega): U \rightarrow
D(\omega)$ is one-to-one almost surely.
\label{setup:Assumption3}
\end{assumption}

\corb{The magnitude of the stochastic domain perturbation is assumed
  to be parameterized as}
\[
\corb{e(x, \omega) := \sum_{n=1}^{N} \sqrt{\mu_{n}}
b_{n}(x)Y_{n}(\omega).}
\]
\corb{Recall that for $n = 1, \dots, N$ let $\Gamma_{n} \equiv
  Y_n(\Omega)$, $\Gamma_{n} \equiv [-1,1]$ and $\Gamma:=
  \prod_{n=1}^{N_s} \Gamma_{N_s}$. Denote $\rho(\bys):\Gamma_s
  \rightarrow \mathbb{R}_{+}$ as the joint probability density of
  $\bY$. Now, the stochastic domain perturbation is split as}
\[
\corb{e(x,\omega) \rightarrow e_{s}(x,\omega) + e_{f}(x,\omega),}
\]
\corb{where we denote $e_{s}(x,\omega)$ as the large deviations and
  $e_{f}(x,\omega)$ as the small deviations modes with the following
  parameterization:}
\[
\corb{e_{s}(x, \omega) := \sum_{n=1}^{N_s} \sqrt{\mu_{s,n}}
b_{s,n}(x)Y_{n}(\omega) \,\,\,\mbox{and} \,\,\, e_{f}(x, \omega) :=
\sum_{n=1}^{N_{f}} \sqrt{\mu_{f,n}} b_{f,n}(x) Y_{n+N_s}(\omega),}
\]
\corb{where $N_s + N_f = N$. Furthermore, for $n = 1,\dots, N_s$ let
  $\mu_{s,n} := \mu_{n}$, $b_{s,n}(x) :=
  b_{n}(x)$, and for $n = 1, \dots, N_f$ let $\mu_{f,n} :=
  \mu_{n+N_s}$ and $b_{f,n}(x) := b_{n+N_s}(x)$.}

\corb{Denote $\bY_s : = [Y_{1}, \dots, Y_{N_s}]$, $\Gamma_s:=
  \prod_{n=1}^{N_s} \Gamma_{n}$, and $\rho(\bys):\Gamma_s \rightarrow
  \mathbb{R}_{+}$ as the joint probability density of $\bY_s$. Similarly
  denote $\bY_f : = [Y_{N_s+1}, \dots, Y_{N}]$, $\Gamma_f:=
  \prod_{n=N_s+1}^{N} \Gamma_{n}$, and $\rho(\byf):\Gamma_f
  \rightarrow \mathbb{R}_{+}$ as the joint probability density of
  $\bY_f$.}
\begin{assumption}~
\corb{
\begin{enumerate}
\item $b_{1},\dots, b_{N} \in W^{2,\infty}(U)$
\item $\| b_{n} \|_{L^{\infty}(U)} = 1$ for $n = 1,2,\dots, N$
\item  $\mu_{n}$ are monotonically decreasing for $n = 1,2,\dots N$.
\item $\mathbb{E}[Y_{n} Y_{m}]= \delta[n-m]$,where $m,n
  = 1, \dots N$.
\end{enumerate}
}
\label{setup:Assumption4}
\end{assumption}

\corb{Now, from the stochastic model the Jacobian $\partial F$ is written as
\begin{equation}
\partial F(x,\omega) = I + 
\sum_{l=1}^{N} B_{l}(x) \sqrt{\mu_{l}}  Y_{l}(\omega) 
\label{analyticity:eqn1}
\end{equation}
where i) for $l = 1,\dots,N_s$, $\sqrt{\mu_l} := \sqrt{\mu_{s,l}}$,
$B_l := B_{s,l}$ and
\[
B_{s,l}(x) := b_{s,l}(x) \partial \hat{v}(x) + 
\hat v(x) \nabla b_{s,l}(x)^{T}
\]
where $\partial v$ is the Jacobian of $v(x)$; ii) for $l =
1,\dots,N_f$ $\sqrt{\mu_{l+N_s}} := \sqrt{\mu_{f,l}}$, $B_{l+N_s} :=
B_{f,l}$ and similar definition for $B_{f,l}$.
\begin{assumption}~
\begin{enumerate}
\item $a \circ F$ and $\hat g$ are only a function of $x
  \in U$ and independent of $\omega \in \Omega$.
\item There exists $0 < \tilde{\delta} < 1$ such that $\sum_{l=1}^{N}
  \| B_l(x) \|_{2} \sqrt{\mu_{l}} \leq 1 -
  \tilde{\delta}$, for all $x \in U$.
\item Assume that $f: \R^d \rightarrow \R$ can be analytically
  extended in $\C^{d}$.
\end{enumerate}
\label{analyticity:assumption1}
\end{assumption}
}

Let ${ H_h(U)}\subset H_0^1(U)$ be the standard finite element space
of dimension $N_h$, which contains continuous piecewise polynomials
defined on regular triangulations $\mathcal{T}_h$ that have a maximum
mesh spacing parameter $h>0$. Let $\hat u_h : \Gamma_s \rightarrow
H_h(U)$ be the semi-discrete approximation that is obtained by
projecting the solution of \eqref{setup:problem1} onto the subspace
$H_h(U)$, for each $\by_s \in \Gamma_s$, i.e.,
\begin{equation}
\begin{split}
\int_U [\Grad \hat u_h(\cdot,\by_s)]^{T} G(\by_s) \Grad v_h\,dx &= \int_U
(f \circ F)(\cdot,\by_s) v_h |\partial F|(\by_s)| \,dx \\ &- L(\hat
\bw,v_h)
\end{split}
\label{collocation-perturbation:eq1}
\end{equation}
for all $v_h\in H_h(U)$ and for a.s.  $\by_s \in \Gamma$. \corb{Note
  that $G(\bys) := (a \circ F(\bys))det(\partial F(\bys))$ $\partial
  F(\bys)^{-1} \partial F(\bys)^{-T}$ and $Q_h(\bys):= Q(\tilde u_h
  \circ F) = Q( \hat{u}_h(\bys))$.}

\section{\corb{Perturbation} }
\label{collocation-perturbation}

\corb{In this section we present a perturbation approach to
  approximate $Q(\by)$ with respect to the domain perturbation. In
  Section \ref{collocation-perturbation:varianceapprox}, the
  perturbation approach is applied with respect to the tail field
  $e_{f}(\cdot,\omega)$. A stochastic collocation approach is then
  used to approximate the contribution with respect to
  $e_{s}(\cdot,\omega)$.}

\corb{Whenever the perturbation of $Q(\by):=Q( (\tilde{u} \circ
  F)(\cdot, \by))$ is small with respect to the parameters ${\bf y}
  \in W$, for a suitable linear vector space $W$ of perturbations, a
  linear approximation is sufficient for an accurate estimate.} To
this end we introduce the following definition.
\begin{definition} \corb{Let $\psi$ be a regular function of the parameters
  $\by \in W$, the Gateaux derivative evaluated at $\by$
  on the space of perturbations $W$ is defined as}
  \[
  <D_{\bf y} \psi({\bf y}), \delta {\bf y} > = \lim_{s \rightarrow 0^{+}}
  \frac{\psi({\bf y} + s \delta {\bf y}) - \psi({\bf y})}{s},\forall
  \delta {\bf y} \in W.
\]
\noindent Similarly, the second order derivative $D^{2}_{{\bf y}}$ as a
bilinear form on $W$ is defined as
  \[
  D^{2}_{{\bf y}} \psi({\bf y})(\delta {\bf y}_{1}, \delta {\bf y}_{2}) =
  \lim_{s \rightarrow 0^{+}} < \frac{  
   D_{{\bf y}} \psi({\bf y} + s \delta
    {\bf y}_{2}) - 
D_{{\bf y}} \psi({\bf y})}{s}, \delta {\bf y}_{1} > ,\,\,\, \forall
  \delta {\bf y}_{2}, \delta {\bf y}_{1} \in W.
  \]
\end{definition}
Suppose that $Q$ is a regular function with respect to the parameters
${\bf y}$, then for all ${\bf y} = {\bf y}_{0} + \delta {\bf y} \in W$
the following expansion holds:
\begin{equation}
  Q({\bf y})  = Q({\bf y}_{0}) + 
  <D_{{\bf y}} Q({\bf y}_{0}), \delta {\bf y} > 
  + \frac{1}{2} D^{2}_{{\bf y}} Q({\bf y} + \theta \delta {\bf y}
  )(\delta {\bf y}, \delta {\bf y})
\label{perturbation:representation}
\end{equation}
\noindent for some $\theta \in (0,1)$. Thus we have a procedure to
approximate the QoI $Q({\bf y})$ with respect to the first order term
and bound the error with the second order term. To explicitly
formulate the first and second order terms we make the following
assumption:

\begin{assumption}
  \corb{For all $v,w \in H^{1}_{0}(U)$, let $\mcG({\bf y}; v, w) :=
    \nabla v^{T}G(\by)\nabla w$, where $G(\by) := (a \circ F)(\cdot,
    \by) \partial F^{-1}({\bf y})$ $\partial F^{-T}({\bf y}) |
    \partial F({\bf y}) |$, we have that for all $\by \in W$}
  \label{perturbation:assumption1}
  \corb{
\begin{enumerate}[(i)]
\item 
$ \nabla_{{\bf y}} \mcG(\by) \in
  [L^{1}(U)]^{N}$
\item For $i = 1, \dots, N$ there exists $\mcC_\mcG(\by) > 0$ s.t.
\[
\int_{U} \partial y_{i} \mcG(\by; v,w)\leq \mcC_\mcG(\by)
\|v\|_{H^{1}_{0}(U)} \|w\|_{H^{1}_{0}(U)}.
\]
\item $\mcC_\mcG(\by)$ is uniformly bounded on $W$.
\end{enumerate}
Furthermore, for all
$\by \in \mcA_{\by}$ we have that
$\nabla_{{\bf y}} (f \circ F)(\by), \nabla_{{\bf y}} \hat \bw(\by) \in 
[L^{1}(U)]^{N}$}
\end{assumption}

\begin{remark} \corb{Although we have that (i) - (iii) are assumptions
    for now, under Assumptions \ref{setup:Assumption1} -
    \ref{setup:Assumption4} and Lemma \ref{errorbounds:lemma3} in
    Section \ref{erroranalysis} it can be shown that Assumption
    \ref{perturbation:assumption1}(i) - (iii) are true for all $\by
    \in \Gamma$.}
  \end{remark}

\begin{definition}
\corb{For all $v,w \in H^{1}_{0}(U)$, and $\by \in W$ let}
\[
<D_{\bf y} B({\bf y}; v,w), \delta {\bf y} > := 
\lim_{s \rightarrow 0^{+}} \frac{1}{s}
[
B({\bf y} + s \delta {\bf y}; v, w) 
-
B({\bf y}; v, w) 
]\,\,\, \forall \delta {\bf y} \in W.
\]
\end{definition}

\begin{remark}
Under Assumption \ref{perturbation:assumption1} for $v,w \in H^1_0(U)$
we have that for all $\by \in W$
\[
<D_{\bf y} B({\bf y}; v,w), \delta {\bf y} >= \int_{U} \nabla_{{\bf
    y}} \mcG({\bf y}; v,w)\cdot \delta {\bf y} = \sum_{n = 1}^{N}
\int_U (\nabla v^T \partial y_n G(\by) \nabla w) \delta y_n
\]
Furthermore, under Assumption \ref{perturbation:assumption1} we have
that
\[
<D_{\bf y} \tf, \delta {\bf y} >= \int_{U} \nabla_{\bf y} \tf \cdot
\delta {\bf y}.
\]
\end{remark}

We can introduce as well the derivative for any function $(v \circ
F)(\cdot,\by) \in L^{2}(U)$ with respect to $\by$: \corb{For all $\by
  \in W$ we have that}
\[
D_{\by} (v \circ F)(\cdot,\by)(\delta \by) := \lim_{s
  \rightarrow o^+} \frac{1}{s}[(v \circ F)(\cdot,\by + s\delta
  \by) - (v \circ F)(\cdot,\by)].
\]
Finally, we assume that Assumptions \ref{setup:Assumption1} \&
\ref{setup:Assumption2} and Problems \ref{setup:Prob3} \&
\ref{setup:Prob4} are valid for the $\R^{N}$ valued vector $\by \in
W$.  This is only to show that the perturbation approach is valid for
the general set of perturbations in $W$. We then use this result in
Section \ref{collocation-perturbation:varianceapprox} for the
allowable perturbations $\by \in \Gamma$.

\begin{lemma} \corb{Suppose that 
    Assumptions \ref{setup:Assumption1}, \ref{setup:Assumption2} and
    \ref{perturbation:assumption1} are satisfied then for any $\by,
    \byd \in W$ and for all $v \in H^{1}_{0}(U)$ we have that}
\[
\begin{split}
& B(\by; D_{\by} (\tilde{u} \circ F)(\cdot,\by)
(\byd), v) 
 = 
\sum_{n = 1}^{N} \delta y_{n} \bigg( \int_{U}  
-(\nabla
\tu)^T
\partial_{y_{n}} G(\by) 
  \nabla v \\
&+
\partial_{y_n} \tf |\partial F(\by)| v 
+ 
\tf \partial_{y_n} |\partial F(\by)|  v \\
&- 
(\nabla \tw)^T  \partial_{y_n}G(\by)  \nabla v
- (\partial_{y_n}
\nabla \tw)^T G(\by) \nabla v \bigg).
\end{split}
\]
\label{perturbation:lemma1}
\end{lemma}
\begin{proof}
  \[
  \corb{
\begin{split}
& B(\by; D_{\by} \tu (\byd), v) 
= 
\lim_{s \rightarrow 0^+} \frac{1}{s} 
\int_{U} ( \nabla \tus ^T \\
&- \nabla \tu^T )G(\by) \nabla v \\
& = 
\lim_{s \rightarrow 0^+} \frac{1}{s}
\int_{U} 
\nabla \tus^T G(\by) \nabla v \\
&- \nabla \tus^T G(\by + s\byd) \nabla v \\
& + 
\lim_{s \rightarrow 0^+} \frac{1}{s}
\int_{U}
\nabla \tus^T G(\by + s\byd) \nabla v 
- \nabla \tu^T G(\by) \nabla v \\
& = - \sum_{i = 1}^{N} 
 \int_{U} (\nabla \tu)^T 
\partial_{y_{i}} G(\by)
  \delta y_{i}
  \nabla v \\
  &+
  \lim_{s \rightarrow 0^+} \frac{1}{s} \left(
  \tilde{l}(\by+s\byd;v) - \tilde{l}(\by;v)
  \right) \\
\end{split}
}
\]
then 
\[
\corb{
\begin{split}
& B(\by; D_{\by} \tu(\byd), v) 
 = \sum_{n = 1}^{N} \int_{U}  
- \nabla \tu^T \partial_{y_n} G(\by) \delta y_n
  \nabla v  \\
&+ \int_U 
\partial_{y_n} \tf \delta y_n |\partial F(\by)| v 
+ \int_U 
\tf \partial_{y_n} |\partial F(\by)| \delta y_n v \\
&- \lim_{s \rightarrow 0^+} \frac{1}{s} \int_U
(\nabla \hat \bw(\by + s\byd))^TG(\by + s\byd)\nabla v 
- (\nabla \hat \bw(\by))^{T} G(\by)\nabla v) \\
\end{split}
}
\]
\corb{The result follows.}
\end{proof}
\begin{lemma}
  \corb{Suppose that Assumptions \ref{setup:Assumption1},
    \ref{setup:Assumption2} and \ref{perturbation:assumption1} are
    satisfied then for any $\by, \byd \in W$ and for all $v \in
    H^{1}_{0}(U)$ we have that}
\[
\begin{split}
B(\by; v, D_{\by} \varphi(\by)
(\byd)   ) 
& = 
\sum_{i = 1}^{N} \int_{U} -(\nabla v)^T
\partial_{y_{i}} G(\by)
  \delta y_{i}
  \nabla \varphi ({\bf y}).  \\
\end{split}
\]
\label{perturbation:lemma2}
\end{lemma}
\begin{proof}
We follow the same procedure as in Lemma \ref{perturbation:lemma1}.
\end{proof}

\begin{remark}
\corb{A consequence of Lemma \ref{perturbation:lemma1} and Lemma
  \ref{perturbation:lemma2} is that if for $n = 1, \dots,N_f$ the
  terms $\|\partial_{y_{n}} G(\by)\|_{2}$, $|\partial_{y_n}
  det(\partial F(\by))|$, $\|\partial_{y_n} (f \circ
  F)(\cdot,\by)\|_{L^{2}(U)}$ and $\|\partial_{y_n}
  \bw(\cdot,\by)\|_{H^{1}_{0}(U)}$ are uniformly bounded for all $\by
  \in W$ then $D_{\by} \tu(\byd)$ and $D_{\by} \varphi(\by)(\byd)$
  belong in $H^{1}_{0}(U)$ for any $\by \in W$ and $\byd \in W$.}
\label{perturbation:remark3}
\end{remark}

\begin{lemma} Under the same assumption as Lemma \ref{perturbation:lemma2}
we have that
\begin{equation}
\begin{split}
& \lim_{s \rightarrow 0^{+}} \frac{Q({\bf y} + s \delta {\bf y}) -
    Q({\bf y})}{s} =
\sum_{i=n}^{N} \delta y_{i} \int_{U} \bigg( - (\nabla \tu )^T 
\partial_{y_{n}} G(\by)
  \nabla \varphi ({\bf y}) \\
&  
+ \partial_{y_n} \tf |\partial F(\by)| \varphi(\by) 
- (\nabla \partial_{y_{n}} {\bf w}(\by))^{T}  G(\by)
\nabla \varphi(\by) \\
&- 
(\nabla {\bf w}(\by))^{T}  \partial_{y_n}G(\by) \nabla \varphi(\by) 
+ 
\tf \partial_{y_n} |\partial F(\by)| \varphi(\by) \bigg).
\end{split}
\label{perturbation:eqn1}
\end{equation}
\noindent where the influence function $\varphi({\bf y})$ satisfies
equation \eqref{setup:dual2}.
\label{perturbation:lemma3}
\end{lemma}
\begin{proof}
\[
\begin{split}
\lim_{s \rightarrow 0^{+}}
& \frac{Q({\bf y} + s \delta {\bf y}) - Q({\bf y})}{s} \\
& = 
\lim_{s \rightarrow 0^{+}} \int_U \frac{1}{s} 
 (\nabla \tus)^T
G(\by+s\byd)
\nabla \varphi(\by+s\byd) \\
&- (\nabla \tu)^T
G(\by)
\nabla \varphi(\by)) \\
&= \sum_{n = 1}^{N} 
\delta y_n \int_U (\nabla \tu)^T 
\partial_{y_{n}} G(\by)
  \nabla \varphi ({\bf y}) \\
&+ \int_U (\nabla D_{\by} \tu
)^T(\byd) G(\by) 
  \nabla \varphi(\by) \\
&+ (\nabla \tu )^T 
G(\by) 
  \nabla D_{\by}\varphi(\by)(\byd). \\
\end{split}
\]
From Lemma \ref{perturbation:lemma1} with $v = \varphi(\by)$
and Lemma \ref{perturbation:lemma2} with $v =  \tu$
we obtain the result.
\end{proof}
\begin{lemma}
  \corb{Suppose that Assumptions \ref{setup:Assumption1},
    \ref{setup:Assumption2} and \ref{perturbation:assumption1} are
    satisfied then for any $\by, \byd \in W$ and for all $v \in
    H^{1}_{0}(U)$ we have that}
\[
\begin{split}
& D^{2}_{{\bf y}} Q({\bf y})(\delta {\bf y}, \delta {\bf y})  
 =
 -\sum_{n,m=1}^{N} \delta y_{n} \delta y_{m}  \bigg(  \int_{U} 
 (\nabla \tu )^T 
(\partial_{y_m} \partial_{y_n}
 G(\by))
\nabla \varphi({\bf y})  \\
&+ 
(\nabla \partial_{y_m} \partial_{y_{n}} {\bf w}(\by))^T
G(\by)\nabla\varphi(\by) + 
(\nabla\partial_{y_{n}} {\bf w}(\by))^{T}  
  \partial_{y_m} G(\by)\nabla\varphi(\by) \\
 &+ 
 (\nabla \partial_{y_m} {\bf w}(\by))^{T}  \partial_{y_n}G(\by)\nabla\varphi(\by) 
+
(\nabla {\bf w}(\by))^{T}  \partial_{y_m} \partial_{y_n} G(\by)\nabla\varphi(\by) \\ 
&-
\partial_{y_m} \partial_{y_n} \tf |\partial F(\by)| \varphi(\by)
- \partial_{y_n} \tf \partial_{y_m} |\partial F(\by)| \varphi(\by)
\\
& - \partial_{y_m} \tf \partial_{y_n} |\partial F(\by)| \varphi(\by)
-  \tf \partial_{y_m} \partial_{y_n} |\partial F(\by)| \varphi(\by)
\bigg)
\\
&- 
\sum_{n=1}^{N} \delta y_n \bigg( \int_{U}
  ( \nabla D_{\by}  
\tu )^{T}(\delta {\bf y})(\partial_{y_n}  G(\by))
\nabla \varphi({\bf y}) \\
&+ 
(\nabla \tu)^{T} (\partial_{y_n}  G(\by))\nabla D_{\by}
\varphi({\bf y}) (\delta {\bf y}) 
+
(\nabla\partial_{y_{n}} {\bf w}(\by))^{T}
  G(\by) \nabla D_{\by}  \varphi(\by)(\byd) \\
&+
(\nabla {\bf w}(\by))^{T}  \partial_{y_n} G(\by) \nabla D_{\by} \varphi(\by)(\byd)
- \partial_{y_n} \tf |\partial F(\by)| D_{\by} \varphi(\by)(\byd) \\
&-  
\tf \partial_{y_n} |\partial F(\by)| D_{\by} \varphi(\by)(\byd)
\bigg).  \\
\end{split}
\label{perturbation:eqn2}
\]
\label{perturbation:lemma5}
\end{lemma}
\begin{proof}
Taking the first variation of equation \eqref{perturbation:eqn1} we
obtain the result.
\end{proof}

\subsection{\corb{Hybrid collocation-perturbation approach}}
\label{collocation-perturbation:varianceapprox}

\corb{We now consider a linear approximation of the QoI Q(\by) with
  respect to $\by$. For any $\by=\byo + \delta \by$, $\byo \in
  \R^{N}$, $\by \in \R^{N}$, the linear approximation has the form
\[
Q^{linear}(\by) := Q(\byo) +  <D_{\by} Q(\byo), \byd > 
\]
where $\byd = \by - \byo \in \R^{N}$. recall that $\Gamma = \Gamma_s
\times \Gamma_f$.  Assume that
\begin{enumerate}[i)]
\item $\by := [\by_{s}, \by_{f}]$, $\delta \by := [\delta
  \by_{s}, \byd_f]$, and $\by_{0} := [\by^{s},
  \by^{f}_{0}]$.
\item $\by^{s}$ takes values on $\Gamma_s$
and $\delta \by_{s} := \bo \in \Gamma_s$.
\item $\by_{0}^{f} := \bo \in \Gamma_f$ and $\byd_f = \byf$ takes
values on $\Gamma_f$.
\end{enumerate}
} We can now construct a linear approximation of the QoI with respect
to the allowable perturbation set $\Gamma$. Consider the following
linear approximation of $Q(\bys,\byf)$
\begin{equation}
  \corb{
    \hat{Q}(\bys, \byf) := 
Q(\bys, \by^f_0) +
\tilde{Q}(\bys, \by^{f}_{0}, \delta \by_{f} )
,}
\label{perturbation:approximation}
\end{equation}
and from Lemma $\ref{perturbation:lemma3}$ we have that
\[
\corb{
\begin{split}
\tilde{Q}(\bys, \by^{f}_{0}, \delta \by_{f} ) := 
 <D_{\by} Q(\bys, \by^{f}_{0}), \byd^{f}_0 > 
& = 
 \sum_{n = 1}^{N_{f}} 
 \delta y^f_n \int_{U} \alpha_n(x,\bys,\by^f_0)\,dx,
\end{split}
}
\]
where
\[
\corb{
\begin{split}
\alpha_{n}(\cdot,\bys,\by^f_0) &:= -(\nabla (\tilde{u} \circ F)(\cdot,\bys,\byo^f)
)^{T}  \partial_{y^f_n} G(\bys,\byo^f) \nabla
\varphi(\bys) \\
&+ \partial_{y^f_n} (f \circ F)(\cdot,\bys,\byo^f) |\partial F(\bys,\byo^f)| \varphi(\bys,\byo^f) \\
& - (\nabla \partial_{y^f_n}
{\bf w}(\bys,\byo^f))^T G(\bys,\byo^f)
\nabla \varphi(\bys,\byo^f) \\
&- 
(\nabla {\bf w}(\bys,\byo^f))^T \partial_{y^f_n}G(\bys,\byo^f) 
\nabla \varphi(\bys,\byo^f) \\
& 
+ (f \circ F)(\cdot, \bys,\byo^f) \partial_{y^f_n} 
|\partial F(\bys,\byo^f)| \varphi(\bys,\byo^f).
\end{split}
}
\]
\corb{
\begin{remark}
It is not hard to see that $<D_{\by} Q(\bys ,
\by^{f}_{0}), \delta \by^{f}_{0} >$ can be rewritten as
\begin{equation}
\begin{split}
 <D_{\by} Q(\bys, \by^{f}_{0}), \byd^{f}_0 > 
& = 
 \sum_{n = 1}^{N_{f}} \sqrt{\mu_{f,n}} \delta y^f_n 
\int_{U} \tilde \alpha_n(x,\bys,\by^f_0)\,dx
\end{split}
\label{perturbation:eqn3}
\end{equation}
where
\[
\begin{split}
\tilde \alpha_{n}(\cdot,\bys,\by^f_0) &:= -(\nabla (\tilde{u} \circ F)(\cdot,
\bys,\byo^f)
)^{T}  \partial_{\tilde y^f_n} G(\bys,\byo^f) \nabla
\varphi(\bys,\byo^f) \\
&+ \partial_{\tilde y^f_n} 
(f \circ F)(\cdot,\bys,\byo^f) |\partial F(\bys,\byo^f)| 
\varphi(\bys,\byo^f) \\
& - (\nabla \partial_{\tilde y^f_n}
{\bf w}(\bys,\byo^f))^T G(\bys,\byo^f)
\nabla \varphi(\bys,\byo^f) \\
&- 
(\nabla {\bf w}(\bys,\byo^f))^T \partial_{ \tilde y^f_n}G(\bys.\byo^f) 
\nabla \varphi(\bys,\byo^f) \\
& 
+ (f \circ F)(\cdot, \bys,\byo^f) \partial_{\tilde y^f_n} |\partial F(\bys,\byo^f)| \varphi(\bys,\byo^f)
\end{split}
\]
and $\tilde y^f_n := y^f_n\sqrt{\mu_{f,n}}$ for $n = 1, \dots, N_f$.
This will allow an explicit dependence of the mean and variance error
in terms of the coefficients $\mu_{f,n}$, $n = 1, \dots, N_f$, as show
in in Section \ref{erroranalysis}.
\label{perturbation:remark}
\end{remark}
}

\corb{The mean of $\hat{Q}(\bys, \byf)$ can be obtained as
\[
\esets{
  \hat{Q}(\bys,
  \byf)
}
=
\esets{
  Q(\bys, \by^f_0)} +
\esets{
\tilde{Q}(\bys, \by^{f}_{0}, \delta \by_{f} )
  }.
\]
From Fubini's theorem we have
\begin{equation}
\begin{split}
  \esets
  {Q(\bys, \byo^f)
  } 
  &=
  \int_{\Gamma_{s}} 
  Q(\bys, \bo)
  \rho_{s} (\by_{s})d\by_{s}.
\end{split}
\label{collocation-perturbation:mean} 
\end{equation}
and from equation \eqref{perturbation:eqn3}
\begin{equation}
\esets{
\tilde{Q}(\bys, \by^{f}_{0}, \delta \by_{f} )
}
=
\sum_{n = 1}^{N_{f}} \sqrt{\mu_{f,n}}
\int_{\Gamma_s} \int_{[-1,1]}
y^f_n 
\gamma_n(\bys,\0)
\rho(\bys, y^f_n)\,d\bys dy^f_n,
\label{perturbation:eqn4}
\end{equation}
where $\gamma_n(\bys,\0) := \int_{U} \tilde \alpha_n(x,\bys,\0)\,dx$,
$\rho(\bys)$ is the marginal distribution of $\rho(\by)$ with respect
to the variables $\bys$ and similarly for $\rho(\bys,y^f_n)$ ($n =
1,\dots,N_f$). The term $\esets{ \tilde{Q}(\bys, \by^{f}_{0}, \delta
  \by_{f} ) }$ is referred as the {\it mean correction}.
}

\corb{The variance of $\hat{Q}(\bys, \byf)$ can be computed as 
\[
  \begin{split}
 \corb{\var[\hat{Q}(\bys,
   \byf)]}
 &=
  \corb{
  \mathbb{E}[\hat{Q}(\bys, \byf)^2] -
  \mathbb{E}[\hat{Q}(\bys, \byf)]^{2} 
=
  \var[
Q(\bys, \by^f_0)
  ]
  }
\\
&
\underbrace{
  \corb{
+
\esets{
\tilde{Q}(\bys,\by^{f}_{0}, \delta \by_{f} )^2
}
  +
  2
  \esets{
    Q(\bys, \by^f_0)\tilde{Q}(\bys,\by^{f}_{0}, \delta \by_{f} )
  }
  }
}_{(I)}
  \\
  &
  \underbrace{
  \corb{
  -
  \esets{
    \tilde{Q}(\bys,\by^{f}_{0}, \delta \by_{f} )
  }^{2}
  -
  2
  \esets{
    Q(\bys, \by^f_0)
  }
  \esets{
     \tilde{Q}(\bys,\by^{f}_{0}, \delta \by_{f} )
  }
  }
  }_{(I)}
  .
\end{split}
  \]
  The term (I) is referred as the {\it variance correction} of
  $\var[Q(\bys, \by_0^f)]$.  From Fubini's theorem and equation
  \eqref{perturbation:eqn3} we have that
\begin{equation}
  \begin{split}
&
    \esets{
\tilde{Q}(\bys,\by^{f}_{0}, \delta \by_{f} )^2
}
    =
\sum_{k = 1}^{N_f} \sum_{n = 1}^{N_f}
\corb{\int_{\Gamma_L} \int_{[-1,1]} \int_{[-1,1]}
  \sqrt{\mu_{f,k}}\sqrt{\mu_{f,n}}
  y^f_k
  y^f_n}
  \\
  &\corb{
    \gamma_j(\bys,\by^f_0)
    \gamma_n(\bys,\by^f_0)
\rho(\bys, y^f_k, y^f_n)\,d\bys dy^f_k dy^f_n
},
\end{split}
\label{perturbation:eqn5}
\end{equation}
and
$ \esets{
    Q(\bys, \by^f_0)\tilde{Q}(\bys,\by^{f}_{0}, \delta \by_{f} )
  }$
is
equal to
\begin{equation}
\sum_{k = 1}^{N_f}  
\int_{\corb{\Gamma_s}} \int_{[-1,1]}
Q(\bys, \0)
\gamma_k(\bys, \0) y^f_k
\rho(\bys, y^f_k)\,\corb{d\bys} dy^f_k.
\label{perturbation:eqn6}
\end{equation}
Note that the mean $\esets {Q(\bys, \byo^f) }$ and variance $\var[
  Q(\bys, \by^f_0) ]$ depend only on the large variation variables
$\bys$.  If the region of analyticity of the QoI with respect to the
stochastic variables $\by_{s}$ is large, it is reasonable to
approximate $Q(\bys, \byo^f)$ with a Smolyak sparse grid
$\mcS^{m,g}_w[Q(\bys, \byo^f)]$. Thus in equations
\eqref{collocation-perturbation:mean} - \eqref{perturbation:eqn6}
$Q(\bys, \by^f_0)$ are replaced with the the sparse grid approximation
$\mcS^{m,g}_w[Q(\bys, \byo^f)]$ and for $n = 1, \dots, N_f$
$\gamma_n(\bys,\0)$ is replaced with $\mcS^{m,g}_w\gamma_n(\bys,\0)$.}

\corb{
\begin{remark} For the special case that $\rho(\by)=\rho(\bys)\rho(\byf)$,
  for all $\bys \in \Gamma_s$ and $\byf \in \Gamma_f$
  (i.e. independence assumption of the joint probability distribution
  $\rho(\bys,\byf)$), the mean and variance corrections are simplified. Applying
  Fubini's theorem and from equation
  \ref{collocation-perturbation:mean} the mean of $\hat{Q}(\bys,
  \byf)$ now becomes
  \[
  \begin{split}
\esets{
  \hat{Q}(\bys,
  \byf)
}
&=
\esets{
  Q(\bys, \by^f_0)} +
\underbrace{ \esets{
\tilde{Q}(\bys, \by^{f}_{0}, \delta \by_{f} )
  }
  }_{=0}
=
\esets{
  Q(\bys, \by^f_0)} \\
&=
  \int_{\Gamma_{s}} 
  Q(\bys, \bo)
  \rho_{s} (\by_{s})d\by_{s},
\end{split}
\]
i.e. there is no contribution from the small variations. Applying a
similar argument we have that
\[
  \begin{split}
    &
    \var[\hat{Q}(\bys,
   \byf)]
 =
  \mathbb{E}[\hat{Q}(\bys, \byf)^2] -
  \mathbb{E}[\hat{Q}(\bys, \byf)]^{2} 
=  
  \var[
Q(\bys, \by^f_0)
  ]
\\
&+
\esets{
\tilde{Q}(\bys,\by^{f}_{0}, \delta \by_{f} )^2
}
  +
  \underbrace{
  2
  \esets{
    Q(\bys, \by^f_0)\tilde{Q}(\bys,\by^{f}_{0}, \delta \by_{f} )
  }
  }_{= 0}
  \\
  &
    \underbrace{
    - 
  \esets{
    \tilde{Q}(\bys,\by^{f}_{0}, \delta \by_{f} )
  }^{2} }_{=0}
  -
  2
  \underbrace{
  \esets{
    Q(\bys, \by^f_0)
  }
  \esets{
     \tilde{Q}(\bys,\by^{f}_{0}, \delta \by_{f} )
  }
  }_{= 0}
  \\
&=
    \var[Q(\bys,
    \by_0^f)]
  +
  \underbrace{
  \sum_{n=1}^{N_{f}}
  \mu^f_n
  \int_{U}
   \alpha_{n}(x,\bys,\by^f_0
 ) \,dx \int_{U} \alpha_{n}(y,\bys,\by^f_0)
   \,dy
  }_{\mbox{Variance correction}}
  .
  \end{split}
  \]
  Notice that for this case the variance correction consists of $N_f$
  terms, thus the computational cost will depend linearly with respect
  to $N_f$.
\end{remark}
}

\section{\corb{Analytic correction}}
\label{analyticity}

\corb{In this section we show that the mean and variance corrections
  are analytic in a well defined region in $\C^{N_s}$ with respect to
  the variables $\by_s$. The size of the regions of analyticity will
  directly correlated with the convergence rate of a Smolyak sparse
  grid. To this end, let us establish the following definition: For
  any $0 < \beta < \tilde{\delta}$, for some constant $\tilde{\delta}
  > 0 $, define the following region in $\C^{N_s}$,
\begin{equation}
\Theta_{\beta,N_s} : = \left\{ {\bf z} \in \mathbb{C}^{N_s};\,
{\bf z} = \by + {\bf w},\,\by \in [-1,1]^{N_s},\,
 \sum_{l=1}^{N_s}  \sup_{x \in U}  \| B_l(x) \|_{2} 
\sqrt{\mu_{l}} |w_{l}| \leq \beta \}
\right\}.
\label{analyticity:region}
\end{equation}
Observe that the size of the region $\Theta_{\beta,N_s}$ is mostly
controlled by the decay of the coefficients $\mu_l$ and the size of
$\| B_l(x) \|_{2}$. Thus the smaller and faster the coefficient
$\mu_l$ decays the larger the region $\Theta_{\beta,N_s}$ will be.}

\corb{Furthermore, rewrite $\partial F(\cdot,\omega)$ as $\partial
  F(\by) = I + R(\by)$, with $R(\by) := \sum_{l=1}^{N}$
  $\sqrt{\mu_{l}}$ $B_{l}(x)y_{l}$. We now state the first
  analyticity theorem for the solution $(\tilde u \circ F)(\bys)$ with
  respect to the random variables $\by \in \Gamma$.}
\begin{theorem} \corb{Let $0 < \tilde{\delta} < 1$ 
then the solution $(\tilde{u} \circ F)(\cdot,\by):\Gamma
\rightarrow H^{1}_{0}(U)$ of Problem \ref{setup:Prob3} can be extended
holomorphically on $\Theta_{\beta,N}$ if}
\[
\beta < min \left\{ \tilde{\delta} \frac{\log{(2 - \gamma)}}{d + \log{(2 -
    \gamma)}}, \sqrt{1 + \tilde{\delta}^2/2} - 1 \right\}
\]
where $\gamma := \frac{2\tilde{\delta}^2 +
  (2-\tilde{\delta})^{d}}{\tilde{\delta}^{d} +
  (2-\tilde{\delta})^{d}}$.
\label{analyticity:theorem1}
\end{theorem}
\begin{proof}
See Theorem 7 in \cite{Castrillon2013}.
\end{proof}
\corb{
\begin{remark}
By following a similar argument, the influence function $\varphi(\by)$
can be extended holomorphically in $\Theta_{\beta,N}$ if
\[
\beta < min \left\{ \tilde{\delta} \frac{\log{(2 - \gamma)}}{d + \log{(2 -
    \gamma)}}, \sqrt{1 + \tilde{\delta}^2/2} - 1 \right\}
\]
\end{remark}
}

\corb{We are now ready to show that the linear approximation $\hat
  Q(\bys,\byf)$ can be analytically extended on $\Theta_{\beta,N_s}$.
  Note that it is sufficient to show that $\int_U \tilde
  \alpha_{n}(\cdot,\by_s,\0)$ can be analytically extended on
  $\Theta_{\beta,N_s}$.
}
\begin{theorem} \corb{Let $0 < \tilde{\delta} < 1$, 
if $\beta < min \{ \tilde{\delta} \frac{\log{(2 - \gamma)}}{d +
  \log{(2 - \gamma)}}, \sqrt{1 + \tilde{\delta}^2/2} - 1 \}$ then
there exists an extension of $\int_U \tilde
\alpha_{n}(\cdot,\by_s,\0)$, for $n = 1, \dots, N_f$, which is
holomorphic on $\Theta_{\beta,N_s}$.}
\label{analyticity:theorem2}
\end{theorem}
\begin{proof}
  \corb{Consider the extension of $\bys \rightarrow \bz_s$, where
    $\bz_s \in \C^{N_s}$. We first show that}
\begin{equation}
\corb{\int_U \nabla (\tilde{u} \circ F)(\by_s,\by_f) )^{T}
\partial_{\tilde y^f_n} G(\by_s,\by_f) \nabla
\varphi(\by_s,\by_f)}
\label{analyticity:eqn2}
\end{equation}
\corb{for $n = 1,\dots,N_f$ can be extended on $\Theta_{\beta,N_s}$.
  Note the for the sake of reducing notation clutter we dropped the
  dependence of the variable $x \in U$ and it is understood from
  context unless clarification is needed.}

\corb{We now show that each entry of the matrix
  $\partial_{\tilde y^f_n}G(\bz_s,\by_f)$ is holomorphic on
  $\Theta_{\beta,N_s}$ for all $\by \in \Gamma_f$. First, we have
  that}
\[
\begin{split}
\partial_{\tilde y^f_n}G(\bz_s,\by_f) 
& =
(\partial_{\tilde y^f_n} (a \circ F)(\bz_s,\by_f) )
C^{-1}(\bz_s,\by_f)
det(\partial F(\bz_s,\by_f)) \\
& + (a \circ F)(\bz_s,\by_f)
\left(
C^{-1}(\bz_s,\by_f)
\partial_{\tilde y^f_n}
det(\partial F(\bz_s,\by_f))
\right. \\
& \left. + det(\partial F(\bz_s,\by_f))
\partial_{\tilde y^f_n} C^{-1}(\bz_s,\by_f)
\right).
\end{split}
\]
From Assumption \ref{analyticity:assumption1} $(a \circ
F)(\cdot,\bz_s)$ and $\partial_{\tilde y^f_l} (a \circ F)(\cdot,\bz_s,\byf) =
0$ \corb{are holomorphic on $\Theta_{\beta,N_s}$ for all $\byf \in
  \Gamma_f$.} From matrix calculus identities we have that
\[
\partial_{\tilde y^f_l} C^{-1}
 (\bz_s,\by_f)
 = - C^{-1} (\bz_s,\by_f)
 \left(
\partial C_{\tilde y_n^f} 
(\bz_s,\by_f)  
 \right) C^{-1}(\bz_s,\by_f).
\]
Since $\beta < \tilde{\delta}$ the series
\[
\partial F^{-1}(\bz_s,\by) 
= (I + R(\bz_s,\by_f) )^{-1} = I +
\sum_{k=1}^{\infty} R(\bz_s,\by_f)
^{k}
\]
is convergent for all ${\bf z_s} \in \Theta_{\beta}$ and for all $
\by_f \in \Gamma_f$. It follows that each entry of $\partial
F(\bz_s,\by)^{-1}$ and therefore $C(\bz_s,\by) ^{-1}$ is holomorphic
\corb{for all $\bz_s \in \Theta_{\beta,N_s}$} and for all $ \by_f \in
\Gamma_f$.  We have that $det(\partial F(\bz_s,\by_f))$ and $\partial
C_{\tilde y_n^f} (\bz_s,\by_f)$ are functions of a finite polynomial
therefore they are holomorphic for all $ \bz_s \in \Theta_{\beta,N_s}$
and $\by_f \in \Gamma_f$.

From Jacobi's formula we have that for all $ \bz_s \in
\Theta_{\beta,N_s}$ and $\by_f \in \Gamma_f$
\[
\begin{split}
\partial_{\tilde y^f_l}
 det(\partial F(\bz_s,\by_f))
& = 
tr(Adj(\partial F(\bz_s,\by_f)) \partial_{y^f_l}
 \partial F(\bz_s,\by_f)
) \\
& = 
 det(\partial F(\bz_s,\by_f))
tr(\partial F(\bz_s,\by_f)^{-1} B^f_n(x)). \\
\end{split}
\]
It follows that for all $ \bz_s \in \Theta_{\beta,N_s}$ and $\by_f \in
\Gamma_f$ 
$\partial_{\tilde y^f_n}G(\bz_s,\by_f)$ are holomorphic.

\corb{We shall now prove the main result. First, extend $\bys$ along
  the $n^{th}$ dimension as $y_n \rightarrow z_n$, $z_n \in \C$ and
  let $\tilde \bz_s = [z_1, \dots, z_{n-1},z_{n+1}, \dots, z_{N_s}]$.
From Theorem \ref{analyticity:theorem1} we have that $(\tilde{u} \circ
F)(\bz_s,\by_f)$ and $\varphi(\bz_s,\by_f)$ are holomorphic for $\bz_s
\in \Theta_{\beta,N_s}$ and $\by_f \in \Gamma_f$ if
\[
\beta < min \{
\tilde{\delta} \frac{\log{(2 - \gamma)}}{d + \log{(2 - \gamma)}},
\sqrt{1 + \tilde{\delta}^2/2} - 1 \}.
\]
Thus from Theorem 1.9.1 in \cite{Gohberg2009} the series
\[
(\tilde{u} \circ F)(\cdot, \bz_s,\by_f)
= 
\sum_{l=0}^{\infty} \tilde{u}_l(\cdot, \tilde \bz_s,\by_f)
z^{l}_n
\,\,\,\mbox{and}\,\,\,
\varphi(\bz_s,\by_f)  
= \sum_{l=0}^{\infty} \bar{\varphi}_l(\cdot, \tilde \bz_s,\by_f)
z^{l}_n,
\]
are absolutely convergent in $H^{1}_{0}(U)$ for all $z \in \C$,
  where $\tilde{u}_l(\cdot,\tilde \bz_s,\by_f),$ $\tilde
  \varphi_l(\cdot, \tilde \bz_s,\by_f) \in H^{1}_{0}(U)$ for $l =
  0,\dots,\infty$. Furthermore,}
\[
\begin{split}
\| \nabla (\tilde{u} \circ F)(\cdot,\bz_s,\by_f)\|_{L^2(U)} 
&\leq
\sum_{l=0}^{\infty} \| \nabla \tilde{u}_l(\cdot,\tilde
\bz_s,\by_f)\|_{L^2(U)}
|z_n|^{l} \\
&\leq
\sum_{l=0}^{\infty} 
\| \tilde{u}_l(\cdot, \tilde \bz_s,\by_f)\|_{H^1_0(U)}
|z_n|^{l}
\end{split}
\]
\corb{i.e. $\nabla (\tilde{u} \circ F)(\cdot,\bz_s,\by_f)$ is
  holomorphic on $\Theta_{\beta,N_s}$ along the $n^{th}$ dimension. A
  similar argument is made for $\nabla \varphi(\cdot,\bz_s,\by_f)$.}

Since the matrix $\partial_{\tilde y^f_n}G(\bz_s,\by_f)$ is holomorphic
for all $ \bz_s \in \Theta_{\beta,N_s}$ and $\by_f \in \Gamma_f$ then
we can rewrite the $(i,j)$ entry as \corb{$\sum_{k=0}^{\infty}
  g^{i,j}_{k}(\cdot, \tilde \bz_s,\by_f) z^{k}_n$ where
  $g^{i,j}_{k}(\cdot, \tilde \bz_s,\by_f) \in L^{\infty}(U)$}.  For
each $i,j = 1, \dots, d$ consider the map
\corb{
\[
\begin{split}
T_{i,j}
&: = 
\int_{U} 
\partial_{\tilde y^f_n}G(\bz_s,\by_f)(i,j)
\partial_{x_i}\tilde{u}(\cdot,\bz_s,\by_f) \partial_{x_j} \varphi
(\cdot,\bz_s,\by_f) \\
&= 
\sum_{k,l,p=0}^{\infty} z_n^{k+l+p} 
\int_{U}
g^{i,j}_{k}
(\cdot,\tilde \bz_s,\by_f)
\partial_{x_{i}}
\tilde{u}_{l}(\cdot,\tilde \bz_s,\by_f)
\partial_{x_{i}} \tilde \varphi_{p}(\cdot,\tilde \bz_s,\by_f).
\end{split}
\]
For $i,j = 1, \dots, d$, for all $ \bz_s \in \Theta_{\beta,N_s}$
and $\byf \in \Gamma_f$}
\[
\begin{split}
|T_{i,j}| 
& \leq
\sum_{k,l,p=0}^{\infty} |z_n|^{k+l+p} 
\int_{U}
|g^{i,j}_{k}
(\cdot,\bz_s,\by_f)
\partial_{x_{i}}
\tilde{u}_{l}(\cdot, \tilde \bz_s,\by_f)
\partial_{x_{i}} \varphi_{p}(\cdot, \tilde \bz_s,\by_f)|
\\
& 
(\mbox{From Cauchy Schwartz it follows that}) \\
& \leq 
\sum_{k,l,p=0}^{\infty} |z_n|^{k+l+p} 
\|g^{i,j}_{k}
(\cdot, \tilde \bz_s,\by_f)\|_{L^{\infty}(U)}
\|
\partial_{x_{i}}
\tilde{u}_{l}(\cdot, \tilde \bz_s,\byf)\|_{L^{2}(U)} \\
&\|
\partial_{x_{i}} \varphi_{p}(\cdot, \tilde \bz_s,\byf)\|_{L^{2}(U)} \\
& \leq 
\sum_{k,l,p=0}^{\infty} |z_n|^{k+l+p} 
\|g^{i,j}_{k}
(\cdot, \tilde \bz_s,\by_f)
\|_{L^{\infty}(U)}
\|
\tilde{u}_{l}(\cdot, \tilde \bz_s,\by_f)\|_{H^1_0(U)} \\
&\|
\varphi_{p}(\cdot, \tilde \bz_s,\by_f)\|_{H^1_0(U)} < \infty.
\end{split}
\]
\corb{Thus equation \eqref{analyticity:eqn2} can be analytically
  extended on $\Theta_{\beta,N_s}$ along the $n^{th}$ dimensions for
  all $\by_f \in \Gamma_f$.  Equation \eqref{analyticity:eqn2} can now
  be analytically extended on the entire domain $\Theta_{\beta, N_s}$.
  Repeat the analytic extension of \eqref{analyticity:eqn2} for $n =
  1, \dots, N_s$.  From Hartog's Theorem it follows that
  \eqref{analyticity:eqn2} is continuous in $\Theta_{\beta,N_s}$. From
  Osgood's Lemma it follows that \eqref{analyticity:eqn2} is
  holomorphic on $\Theta_{\beta,N_s}$. Following a similar argument as
  for \eqref{analyticity:eqn2} we can analytically extended the rest
  of the terms of $\alpha_{n}(\cdot, \by_s, \by_f)$ on
  $\Theta_{\beta,N_s}$ for $n = 1,\dots,N_f$.}
\end{proof}

\section{Error analysis}
\label{erroranalysis}
In this section we analyze the error between the exact QoI
$Q(\bys,\byf)$ and the sparse grid hybrid perturbation
approximation $\mcS^{m,g}_w [\hat{Q}_h(\by_{s}, \by_{f}) ]$.  With a
slight abuse of notation by $\mcS^{m,g}_w [\hat{Q}_{h}(\by_{s},
  \by_{f}) ]$ we mean the two sparse grids approximations:
\[
\corb{\mcS^{m,g}_w [\hat{Q}_{h}(\by_{s}, \by_{f}) ] := \mcS^{m,g}_w
  [Q_{h}(\by_s,\bo)] + \sum_{n=1}^{N_f} \sqrt{\mu_{f,n}} y^{f}_{n}
  \mcS^{m,g}_w [\int_U \tilde \alpha_{n,h}(\cdot,\by_s,\bo)],}
\]
\corb{where $\alpha_{n,h}(\cdot,\by_s,\bo)$, for $n = 1, \dots, N_f$,
  and $Q_h(\bys,\byf)$ are the finite element approximations of
  $\alpha_{n}(\cdot,\by_s,\bo)$ and $Q(\bys,\byf)$ respectively. It is
  easy to show that $\var(Q(\bys,\byf)) -
  \var(\mcS^{m,g}_{\lv}[\hat{Q}_{h}(\bys,\byf)])$ is equal to}
\[
\underbrace{
  \mathbb{E}[Q^{2}(\bys,\byf)
 -
 \mcS^{m,g}_{\lv}[\hat{Q}_{h}(\bys,\byf)]^2]
}_{(I)} -
\underbrace{(\mathbb{E}[Q(\bys,\byf)]^{2} - 
\mathbb{E}[\mcS^{m,g}_{\lv}[\hat{Q}_{h}(\bys,\byf) ]^{2}])
}_{(II)}.
\]
\corb{
\noindent (I) Applying Jensen's inequality we have that
\begin{equation}
\begin{split}
& 
|\mathbb{E}[Q^{2}(\by)
 -
\mcS^{m,g}_{\lv}[\hat{Q}_{h}(\by)]
^2]| 
\leq
\| Q(\by) + \mcS^{m,g}_{\lv}[\hat{Q}_{h}(\by)]
\|_{L^{\infty}_{\rho}(\Gamma)} 
( \| Q(\by)
  - \hat{Q}(\by) \|_{L^2_{\rho}(\Gamma)} \\
& +
\|\hat{Q}(\by) -
  \hat{Q}_{h}(\by)\|_{L^1_{\rho}(\Gamma)}
 +
  \| \hat{Q}_{h}(\by) -
    \mcS^{m,g}_{\lv}[\hat{Q}_{h}(\by)]]\|_{L^2_{\rho}(\Gamma)}).
\end{split}
\label{erroranalysis:eqn1}
\end{equation}
(II) Similarly, we have that
\[
\begin{split}
& 
|\mathbb{E}[Q(\by_s,\by_f)]^{2} - \mathbb{E}[\mcS^{m,g}_{\lv}
[\hat{Q}_{h}(\by_s,\by_f)
]]^{2}|
\leq 
\|Q(\by) + \mcS_{w}^{m,g} \hat{Q}_{h}(\by) \|_{L^1_{\rho}(\Gamma)} \\
& 
\|Q(\by)
- \mcS_{w}^{m,g} \hat{Q}_{h}(\by_{s}) \|_{L^1_{\rho}(\Gamma)} \\
&\leq 
\|Q(\by) + \mcS_{w}^{m,g} \hat{Q}_{h}(\by) \|_{L^1_{\rho}(\Gamma)}
(\|Q(\by) - \hat{Q}(\by) \|_{L^1_{\rho}(\Gamma)} 
+ \|\hat{Q}(\by) - \hat{Q}_{h}(\by) \|_{L^1_{\rho}(\Gamma)} \\
&
+ \|\hat{Q}_{h}(\by) - \mcS_{w}^{m,g} \hat{Q}_{h}(\by) \|_{L^1_{\rho}(\Gamma)})
\end{split}
\]
Applying Jensen inequality
\begin{equation}
  \begin{split}
& 
|\mathbb{E}[Q(\by_s,\by_f)]^{2} - \mathbb{E}[\mcS^{m,g}_{\lv}
[\hat{Q}_{h}(\by_s,\by_f)
]]^{2}|    
\leq 
\|Q(\by) + \mcS_{w}^{m,g} \hat{Q}_{h}(\by) \|_{L^2_{\rho}(\Gamma)} \\
&(
\|Q(\by) - \hat{Q}(\by) \|_{L^2_{\rho}(\Gamma)} 
+ \|\hat{Q}(\by) - \hat{Q}_{h}(\by) \|_{L^1_{\rho}(\Gamma)} 
+ 
\|\hat{Q}_{h}(\by) - \mcS_{w}^{m,g} \hat{Q}_{h}(\by) \|_{L^2_{\rho}(\Gamma)}).\\
\end{split}
\label{erroranalysis:eqn2}
\end{equation}
\corb{Combining equations \eqref{erroranalysis:eqn1} and
  \eqref{erroranalysis:eqn2} we have that
\[
\begin{split}
&|\var(Q(\by)) -
\var(\mcS^{m,g}_{\lv}[\hat{Q}_{h}(\by)])| 
\leq
C_P\underbrace{\| Q(\by) - \hat{Q}(\by)
    \|_{L^2_{\rho}(\Gamma)}}_{\mbox{Perturbation}} \\
&+ 
C_{PFE}\underbrace{\|
    \hat{Q}(\by) - \hat{Q}_{h}(\by) \|_{L^1_{\rho}(\Gamma)}}_{\mbox{Finite
      Element}} 
+ 
  C_{PSG} \underbrace{
  \| \hat{Q}_h(\by) - \mcS^{m,g}_w[\hat{Q}_h(\by)] \|_{L^2_{\rho}(\Gamma)}
}_{\mbox{Sparse Grid}}.
\end{split}
\]
Similarly we have that the mean error satisfies the following bound:
\[
\begin{split}
&|\mathbb{E}[Q(\bys,\byf)
 -
 \mcS^{m,g}_{\lv}[\hat{Q}_{h}(\bys,\byf)]|
\leq
\underbrace{\| Q(\by) - \hat{Q}(\by)
    \|_{L^2_{\rho}(\Gamma)}}_{\mbox{Perturbation (I)}} \\
&+ 
\underbrace{\|
    \hat{Q}(\by) - \hat{Q}_{h}(\by) \|_{L^1_{\rho}(\Gamma)}}_{\mbox{Finite
      Element (II)}} 
+ 
\underbrace{
  \| \hat{Q}_h(\by) - \mcS^{m,g}_w[\hat{Q}_h(\by)] \|_{L^2_{\rho}(\Gamma)}
}_{\mbox{Sparse Grid (III)}}.
\end{split}
\]
  }
  }
\corb{
  \begin{remark} For the case that probability distributions
    $\rho(\bys)$ and $\rho(\byf)$ are independent then the mean
    correction is exactly zero, thus the mean error would be bounded
    by the following terms
\[
\begin{split}
&|\eset{Q(\by_s,\by_f)} -
\eset{\mcS^{m,g}_w[Q_{h}(\by_s)}]| 
\leq
C_{T} \underbrace{\| Q(\by_{s},\by_{f})
 - Q(\by_{s})
  \|_{L^2_{\rho}(\Gamma)}}_{\mbox{Truncation}} \\
&+
C_{FE}\underbrace{\|Q(\by_{s}) - Q_{h}(\by_{s})
  \|_{L^1_{\rho}(\Gamma_s)}}_{\mbox{Finite Element} } 
+ 
C_{SG}
\underbrace{
  \| Q_{h}(\by_{s}) - \mcS^{m,g}_w[Q_{h}(\by_{s})] 
\|_{L^2_{\rho}(\Gamma_s)}
}_{\mbox{Sparse Grid}}
\end{split}
\]
\noindent for some positive constants $C_{T},C_{FE}$ and $C_{SG}$. We
refer the reader to Section 5 in \cite{Castrillon2013} for the
definition of the constants and bounds of these errors.
  \end{remark}
  }
\subsection{Perturbation error}
\label{erroranalysis:perturbation}
In this section we analyze the error term (I):
\corb{
\begin{equation}
\| Q(\bys,\byf) - \hat{Q}(\bys,\byf) \|_{L^2_{\rho}(\Gamma)}
 =
\| \mcR(\bys,\byd_f) \|_{L^2_{\rho}(\Gamma)}  
\label{errorbounds:perturbation}
\end{equation}
where $\byf = \byo^f + \byd_f$, $\byo^f = \0$ and the remainder is
equal to
\begin{equation}
\mcR(\bys,\byd_f) :=
\frac{1}{2} D^{2}_{\byf} Q(\bys + \theta \byd_f
)(\byd_f, \byd_f)
\label{errorbounds:residual}
\end{equation}
\noindent for some $\theta \in (0,1)$.  
From the approximation $\hat{Q}(\bys, \byf)$ in equation
\eqref{perturbation:approximation} and the expansion from
\eqref{perturbation:representation} it is clear that
\[
|Q(\bys,\byf) - \hat{Q}(\bys, \byf)| \leq 
\frac{1}{2} D^{2}_{\byf} Q(\bys + \theta \byd_f
)(\byd_f, \byd_f).
\]
We shall now prove a series of lemmas that will be used to bound the
perturbation error.}

\corb{Recall from Remark \ref{perturbation:remark} that we use the
  approximation given by equation \eqref{perturbation:eqn3}
\[
\begin{split}
 <D_{\by} Q(\bys, \by^{f}_{0}), \byd^{f}_0 > 
& = 
 \sum_{n = 1}^{N_{f}} \sqrt{\mu_{f,n}} \delta \tilde y^f_n 
\int_{U} \tilde \alpha_n(x,\bys,\by^f_0)\,dx
\end{split}
\]
where the variable dependence is on $\tilde y^{f}_{n} =
\sqrt{\mu_{f,n}} y^{f}_{n}$ instead of $y^{f}_{n}$ for $n =
1,\dots,N_f$. This will allow an explicit dependence of the mean and
variance error on decay parameters $\mu^{f}_{n}$ of the tail. To make
the exposition clearer we use the following notation, let
\[
\tilde \by_f :=
\begin{bmatrix}
  \tilde y^{f}_1 \\
  \vdots \\
  \tilde y^{f}_{N_s}
\end{bmatrix}
=
\begin{bmatrix}
  \sqrt{\mu_{f,1}} y^{f}_1 \\
  \vdots \\
  \sqrt{\mu_{f,N_f}} y^{f}_{N_s}
\end{bmatrix}
,
\tilde \by :=
\left[
\begin{array}{c}
  \by_s \\
  \hline
  \tilde \by_f
\end{array}
\right]
\]
and for all $\by_f \in \Gamma_f$ we have that $\delta \tilde \by_f :=
\tilde \by_f$.
}
\corb{\begin{lemma} For all $n = 1, \dots, N_{f}$ and for all $ \by \in \Gamma$
\[
\sup_{x \in U}
\sigma_{max} \left(
\partial_{\tilde y^{f}_{n}}
\partial F^{-1}(\by) \right)
\leq \sup_{x \in U} \| B_{f,n}(x) \|_2 
\F^{-2}_{min} 
\]
\label{errorbounds:lemma1}
\end{lemma}
}
\begin{proof} 
From matrix calculus we have
\[
\partial_{\tilde y^{f}_{n}}
 \partial F^{-1}(\by) = - \partial
F^{-1}(\by)\left( \partial_{\tilde y^{f}_{n}}
\partial F(\by)  \right)
\partial F^{-1}(\by)
\]
\noindent and also
\corb{
\[
\sigma_{max}\left( \partial_{\tilde y^{f}_{n}} \partial
F(\by) \right) \leq 
\sigma_{max}(B_{f,n}(x)).
\]
}
\noindent From Assumption \ref{setup:Assumption1} the result follows.
\end{proof}
\corb{\begin{lemma} For all $\by \in \Gamma$
\[
\sup_{x \in U} |\partial_{\tilde y^{f}_{n}} det(\partial F(\by))|
\leq \sup_{x \in U} 
\F_{max}^{d}\F^{-1}_{min}\|B_{f,n}(x)\|_{2}d
\]
\label{errorbounds:lemma2}
\end{lemma}
\begin{proof} 
Using Jacobi's formula we have that for all $\by \in \Gamma$
\[
\begin{split}
\partial_{\tilde y^{f}_{n}}
det(\partial F(\by))
& = tr(Adj(\partial F(\by)) 
\partial_{\tilde y^{f}_{n}}
\partial F(\by)   ) \\
& = 
det(\partial F(\by))
\sum_{i=1}^{d}  
\lambda_{i}(
\partial F(\by)^{-1} B_{f,n}(x))
),
\end{split}
\]
where $\lambda_{i}(\cdot)$ are the eigenvalues.
\end{proof}}
\corb{\begin{lemma} For all $n,m = 1, \dots, N_{f}$ and for all $\by \in \Gamma$
\[
\sup_{x \in U} \sigma_{max} \left(
\partial_{\tilde y^{f}_{n}} \partial_{\tilde y^{f}_{m}}
\partial F^{-1}(\by)
\right) 
\leq \sup_{x \in U}
2 \F_{min}^{-3}
\|B_{f,n}(x)\|_2 \|B_{f,m}(x)\|_2.
\]
\label{errorbounds:lemma3}
\end{lemma}}
\begin{proof} Using matrix calculus identities we have that
\[
\begin{split}
\partial F^{-1}(\by)
&= 
\partial F^{-1}(\by) 
\left[
\left(
\partial_{\tilde y^{f}_{m}}
\partial F(\by) \right)
\partial F^{-1}(\by)
 \left(
 \partial_{\tilde y^{f}_{n}}
 \partial F(\by) \right) \right. \\
 &
 + 
 \left.
 \left( 
 \partial_{\tilde y^{f}_{n}} \partial
 F(\by)  \right) 
 \partial F^{-1}(\by)  \left(
 \partial_{\tilde y^{f}_{m}}
 \partial F(\by)  \right) 
\right] \partial F^{-1}(\by). 
\end{split}
\]
Taking the triangular and multiplicative inequality, and following the
same approach as Lemma \ref{errorbounds:lemma1} we obtain the desired
result.
\end{proof}
\corb{\begin{lemma} For $n,m = 1,\dots,N_f$ for all $\by \in \Gamma$
\[
\sup_{x \in U}
\sigma_{max} \left( 
\partial_{\tilde y^{f}_{m}} \partial_{\tilde y^{f}_{n}}  | \partial F(\by)|
  \right)
\leq 
\sup_{x \in U} 
d(d+1) \F^{d}_{max}\F^{-2}_{min}\|B_{f,n}(x)\|_2\|B_{f,m}(x)\|_2.
\]
\label{errorbounds:lemma4}
\end{lemma}}
\begin{proof}
Using Jacobi's formula we have
\[
\begin{split}
\partial_{\tilde y^{f}_{m}} 
\partial_{\tilde y^{f}_{n}}
& = 
\partial_{\tilde y^{f}_{m}} (
det(\partial F(\by))
tr( 
\partial F(\by)^{-1} B_{f,n}(x)))
\sqrt{\mu_{f,n}}
\\
& = 
\partial_{\tilde y^{f}_{m}} 
det(\partial F(\by))
tr( 
\partial F(\by)^{-1} B_{f,n}(x)) \\
&+ 
det(\partial F(\by))
tr(\partial_{\tilde y^{f}_{m}} 
\partial F(\by)^{-1} B_{f,n}(x)
) \\
& =
det(\partial F(\by))
tr(
\partial F(\by)^{-1} 
B_{f,m}(x))
tr(
\partial F(\by)^{-1} 
B_{f,n}(x)
) \\
& -
det(\partial F(\by))
tr(
\partial F(\by)^{-1} 
B_{f,m}(x)
\partial F(\by)^{-1} 
B_{f,n}(x) 
).
\end{split}
\]
The result follows.
\end{proof}
\begin{lemma}
  \corb{For all $v,w \in H^{1}_{0}(U)$, $\theta \in (0,1)$,
    $\bys \in \Gamma_s$ and $\byd_f \in
    \Gamma_f$ we have that}
\[ 
\begin{split}
&| \int_{U} 
(a \circ F)(\cdot,\bys,\bo)
(\nabla v)^{T}
  \partial_{\tilde y^{f}_{n}} G(\bys + \theta \delta \tilde \by_f )
\nabla w | \leq
\| v \|_{H^{1}_{0}(U)} 
\| w \|_{H^{1}_{0}(U)} 
a_{max} \\
&\mcB(d,\F_{min}, \F_{max},B_{f,n}),
\end{split}
\]
where
\[
\mcB(d,\F_{min}, \F_{max},B_{f,n}) 
:= 
\sup_{x \in U}
(d +
2)\F^{d}_{max}\F^{-3}_{min}\|B_{f,n}(x)\|_2.
\]
\label{errorbounds:lemma5}
\end{lemma}
\begin{proof} \corb{First we expand the partial derivative of $G(\by)$
  with respect to $\tilde y^{f}_{n}$:}
\[
\begin{split}
  \partial_{\tilde y^{f}_{n}} 
G(\by)
& =   \partial_{\tilde y^{f}_{n}} \partial F^{-T}(\by)
 \partial F^{-1}(\by) |\partial
F(\by)| +
\partial F^{-T}(\by) 
  \partial_{\tilde y^{f}_{n}} 
\partial
F^{-1}(\by) |\partial F(\by)| \\
& + \partial F^{-T}(\by) \partial F^{-1}(\by) 
  \partial_{\tilde y^{f}_{n}} 
|\partial F(\by)|,
\end{split}
\]
From Lemmas \ref{errorbounds:lemma1}, \ref{errorbounds:lemma2}
and the triangular inequality we have that 
\[
\begin{split}
\sup_{x \in U, \by \in \Gamma} \sigma_{max} \left(
\partial_{\tilde y^f_n}
 G(\by)
\right) & \leq (d + 2)\F^{d}_{max}\F^{-3}_{min}\|B_{f,n}(x)\|_2.
\end{split}
\]
\end{proof}
\begin{lemma}
\corb{For all $v,w \in H^{1}_{0}(U)$, $\theta \in (0,1)$ and $\byd_f
  \in \Gamma_f$ we have that}
\[
\left| \int_{U} (a \circ F)(\cdot,\bys) (\nabla w)^{T} 
\partial_{\tilde y^{f}_{n}}  \partial_{\tilde y^{f}_{m}} 
G (\bys + \theta   \delta \tilde \by )
\nabla v \right|
\]
is less or equal to 
\[
\| v \|_{H^{1}_{0}(U)} \| w \|_{H^{1}_{0}(U)} a_{max} 
2(d + 3)\F_{min}^{-4} \F_{max}^d \|B_{f,n}(x)\|_2 \|B_{f,m}(x)\|_2.
\]
\label{errorbounds:lemma6}
\end{lemma}
\begin{proof}
\corb{Using matrix calculus identities we have that for all $ \by \in \Gamma$}
\[
\begin{split}
\partial_{\tilde y^{f}_{m}}  \partial_{\tilde y^{f}_{n}} G (\by) 
& =   \partial_{\tilde y^{f}_{m}} \partial_{\tilde y^{f}_{n}} \partial F^{-T}(\by)
 \partial F^{-1}(\by) |\partial
F(\by)| \\
& + \partial_{\tilde y^{f}_{n}}
\partial F^{-T}(\by) 
\partial_{\tilde y^{f}_{m}} 
\partial
F^{-1}(\by) |\partial F(\by)| \\
& + \partial_{\tilde y^{f}_{n}} \partial F^{-T}(\by) \partial F^{-1}(\by) 
  \partial_{\tilde y^{f}_{m}} 
|\partial F(\by)| \\
& +   
\partial_{\tilde y^{f}_{m}}
\partial F^{-T}(\by) 
\partial_{\tilde y^{f}_{n}} 
\partial
F^{-1}(\by) |\partial F(\by)| \\
&+ 
\partial F^{-T}(\by) \partial_{\tilde y^{f}_{m}} \partial_{\tilde y^{f}_{n}} 
 \partial F^{-1}(\by) |\partial
F(\by)| 
\\
& + 
\partial F^{-T}(\by) \partial_{\tilde y^{f}_{n}} \partial F^{-1}(\by) 
  \partial_{\tilde y^{f}_{m}} 
|\partial F(\by)| \\
& +   
\partial_{\tilde y^{f}_{m}}
\partial F^{-T}(\by) 
\partial
F^{-1}(\by) \partial_{\tilde y^{f}_{n}} |\partial F(\by)| \\ 
&+ 
\partial F^{-T}(\by) \partial_{\tilde y^{f}_{m}} 
 \partial F^{-1}(\by) 
\partial_{\tilde y^{f}_{n}} 
|\partial
F(\by)| 
\\
& + 
\partial F^{-T}(\by) 
\partial F^{-1}(\by) 
  \partial_{\tilde y^{f}_{m}} 
\partial_{\tilde y^{f}_{n}}  
|\partial F(\by)|.\\
\end{split}
\]
\corb{From Lemmas \ref{errorbounds:lemma1}, \ref{errorbounds:lemma2},
  \ref{errorbounds:lemma3}, \ref{errorbounds:lemma4}, and the
  triangular inequality we have that for all $ \by \in \Gamma$}
\[
\begin{split}
\| \partial_{\tilde y^{f}_{m}}  \partial_{\tilde y^{f}_{n}} G(\by) \|_2 
\leq & \sup_{x \in U}
(7 + 3d + 2 \F_{min}^{-1}\F_{max}^{-d})
\F_{min}^{-4} \F_{max}^d \|B_{f,n}(x)\|_2 \|B_{f,m}(x)\|_2 
\end{split}
\]
\end{proof}

\begin{assumption}
\corb{For all $\by \in \Gamma$ we assume that $(f \circ F)(\cdot,\by)
  \in H^2(U)$.}
\label{errorbounds:assumption1}
\end{assumption}
\corb{\begin{lemma} For all $\bys \in \Gamma_s$ and $\byd_f \in \Gamma$ we 
    have that:\\
\noindent (a)
\[
\begin{split} 
&\| \nabla D_{\byf} \tilde{u}(\cdot,\bys,\bo)(\byd_f)\|_{L^2(U)} 
\leq
\frac{\sup_{x \in U} 
\sum_{n = 1}^{N_f} \sqrt{\mu_{f,n}}}{a_{min}\F^d_{min}\F^{-2}_{max}}
\\
&\Big( (\|  (\tilde{u} \circ F)(\cdot,\bys,\bo) \|_{H^{1}_{0}(U)} 
+
\| \hat \bw \|_{H^{1}(U)}) 
a_{max}
(d +
2)\F^{d}_{max}\F^{-3}_{min} \|B_{f,n}(x)\|_2 \\
&+ 
\F_{max}^d \F_{min}^{-1}
\|\hat{v}\|_{[L^{\infty}(U)]^d}  \|b_{f,n}\|_{L^{\infty}(U)}
C_P(U)\sqrt{d} 
\|  (f \circ F)(\cdot,\bys,\bo) \|_{H^{1}(U)}\\
&+ 
d C_P(U)\|  (f \circ F)(\cdot,\bys,\bo) \|_{L^{2}(U)} \F_{max}^d 
\F^{-1}_{min} \|B_{f,n}(x)\|_2 \\
\end{split}
\]
\bigskip
\noindent (b)
\[
\begin{split} 
&\| \nabla D_{\byf} \varphi(\bys,\bo)(\byd_f)\|_{L^2(U)} \\
&\leq
\frac{\sup_{x \in U} 
\sum_{n = 1}^{N_f} \sqrt{\mu_{f,n}}
\|  \varphi(\bys,\bo) \|_{H^{1}_{0}(U)}
a_{max}
(d +
2) \|B_{f,n}\|_2
}
{a_{min}\F^d_{min}\F^{-2}_{max}\F^{-d}_{max}\F^{3}_{min} }.
\end{split}
\]
\label{errorbounds:lemma7}
\end{lemma}
\begin{proof}
  (a) From Lemma \ref{perturbation:lemma1}, Remark
  \ref{perturbation:remark3}
  \&  \ref{perturbation:remark}
  we have that for any $v \in
  H^{1}_0(U)$ and for all $\bys \in \Gamma_s$ and $\byd_f \in \Gamma$
\[
\begin{split} 
&a_{min}\F^d_{min}\F^{-2}_{max}
\| \nabla D_{\by_f} \tilde{u}(\cdot,\bys,\bo)(\byd_f)\|_{L^2(U)} 
\| \nabla v \|_{L^2(U)} \\
&\leq
\sum_{n = 1}^{N_f} \sqrt{\mu_{f,n}} \delta y^f_{n} \bigg| \int_{U}  
-\nabla (\tilde{u} \circ F)(\cdot,\bys,\bo)^{T} 
\partial_{\tilde y^f_{n}} G(\bys,\bo) 
  \nabla v \\
&+
  (\partial_{\tilde y^{f}_{n}} (f \circ F)(\cdot,\bys,\bo))
  |\partial F(\bys,\bo)| v + 
  (f \circ F)(\cdot,\bys,\bo) \partial_{\tilde y^{f}_{n}}
  |\partial F(\bys,\bo)| v \\
&- 
(\nabla \hat \bw(\cdot,\bys,\bo))^{T}  \partial_{\tilde y^{f}_{n}}
G(\bys,\bo)  \nabla v 
- (
\nabla \partial_{\tilde y^f_n}\hat \bw(\cdot,\bys,\bo))^{T}
G(\bys,\bo) \nabla v \bigg|.
\end{split}
\]
With the choice of $v = D_{\by_f} \tilde{u}(\cdot,\bys,\bo)(\byd_f)$
and from Lemma \ref{errorbounds:lemma5} we have that
\[
\begin{split} 
&\| \nabla D_{\byf} \tilde{u}(\cdot,\bys,\bo)(\bys,\byd)\|_{L^2(U)} 
\leq
\frac{1}{a_{min}\F^d_{min}\F^{-2}_{max}
\| \nabla v \|_{L^2(U)} } \bigg( \\
&\sum_{i = 1}^{N} \sqrt{\mu_{f,n}}
\delta y^f_{n}
\bigg| \int_{U}  
-(\nabla (\tilde{u} \circ F)(\cdot,\bys,\bo))^{T} \partial_{\tilde
  y^f_{n}} G(\bys,\bo) 
  \nabla v \\
&+
\partial_{\tilde y^{f}_{n}} (f \circ F)(\cdot,\bys,\bo) |\partial F(\bys,\bo)| v 
+ 
(f \circ F)(\cdot,\bys,\bo) \partial_{\tilde y^{f}_{n}} |\partial F(\bys,\bo)| v \\
&- 
(\nabla {\bf w}(\bys,\bo))^{T}  \partial_{\tilde y^{f}_{n}}G(\bys,\bo)  \nabla v 
- (\partial_{\tilde y^{f}_{n}}
\nabla {\bf w}(\bys,\bo))^{T} G(\bys,\bo) \nabla v \bigg|\bigg). 
\end{split}
\]
Now,
\[
\begin{split}
&\Big|
\int_U \partial_{\tilde y^{f}_{n}} (f \circ F)(\cdot,\bys,\bo) 
|\partial F(\bys,\bo)| v \Big| 
\leq \F_{max}^d
\int_U |\partial_{\tilde y^{f}_{n}} (f \circ F)(\cdot,\bys,\bo)v| \\
&\leq
\F_{max}^d C_{P}(U) \| \nabla v \|_{L^{2}(U)}
\| \sum_{l=1}^d |\partial_{F_l} f \partial_{\tilde y^{f}_{n}} F_l| \|_{L^{2}(U)}
\\
&\leq \F^d_{max} C_{P}(U) \| \nabla v \|_{L^2(U)}
\|\hat{v}\|_{[L^{\infty}(U)]^d} \|b_{f,n}\|_{L^{\infty}(U)} 
\\
&\| \nabla f \cdot {\bf 1} \|_{L^{2}(U)}.
\end{split}
\]
From the Sobolev chain rule (see Theorem 3.35 in \cite{Adams1975}) for
any $v \in H^{1}(D(\omega))$ we have that $\nabla v = \partial F^{-T}
\nabla (v \circ F)$, where $v \circ F \in H^{1}(U)$, thus
\[
\begin{split}
\| \nabla f \cdot {\bf 1} \|_{L^{2}(U)} 
&= 
\| {\bf 1}^{T} \partial F^{-T} \nabla (f \circ F)(\cdot,\bys,\bo)
 \|_{L^{2}(U)} \\
&\leq 
\F_{min}^{-1} \sqrt{d} 
\| \| \nabla (f \circ F)(\cdot,\bys,\bo) \|_{2} \|_{L^{2}(U)} \\
&\leq 
\F_{min}^{-1} \sqrt{d} 
\| (f \circ F)(\cdot,\bys,\bo) \|_{H^{1}(U)}.
\end{split}
\]
From Lemma \ref{errorbounds:lemma2} the result follows.\\
\\
\noindent (b) Apply Lemma \ref{perturbation:lemma2} with 
$v =
D_{\byf} \varphi(\bys,\bo)(\byd_f)$
and we get the result.
\end{proof}}
\corb{\begin{lemma}
For all $\by \in \Gamma$ and $n,m = 1,\dots,N$ we have that
\[
\begin{split}
&\| \partial \tilde y^{f}_{n} (f \circ F)(\cdot,\by) \|_{L^2(U)} \\
&\leq
  \F_{min}^{-1} \sqrt{d} \| b_{f,n} \|_{L^{\infty}(U)} \| \hat{v}
  \|_{[L^{\infty}(U)]^d} \| (f \circ F)(\cdot,\by) \|_{H^1(U)}.
\end{split}
\]
and
\[
\begin{split}
&\partial \tilde y^{f}_{n} \partial y^f_m (f \circ F)(\cdot,\by)
  \|_{L^2(U)} 
\leq 
d \| b_{f,n} \|_{L^{\infty}(U)} \| b_{f,m}
  \|_{L^{\infty}(U)} \| \hat{v} \|^2_{[W^{1,\infty}(U)]^d} \\
&(d
  \F^{-2}_{min}( d^{3/2} \| f \circ F \|_{H^{1}(U)} \F^{-2}_{min}(1 +
  4 \| \hat{v} \|_{[L^{\infty}(U)]^d} ( \sum_{i = 1}^{N_s}
  \sqrt{\mu}_{s,i} \| b_{s,i} \|_{W^{2,\infty}(U)} \\
&+ \sum_{i =
    1}^{N_f} \sqrt{\mu}_{f,i} \| 
b_{f,i} \|_{W^{2,\infty}(U)}))) + \|
  f \circ F \|_{H^{2}(U)}).
\end{split}
\]
\label{errorbounds:lemma10}
\end{lemma}}
\begin{proof}~
\bigskip
\corb{The first bound is immediate. Follow the proof in Lemma
\ref{errorbounds:lemma7} (a). Now, by applying the chain rule for
Sobolev spaces we obtain that for all $\by \in \Gamma$}
\begin{equation}
\begin{split}
\| \partial \tilde y^f_n \partial \tilde y^f_m (f \circ F)(\cdot,\by)
\|_{L^2(U)} 
&=
\| 
(
b_m b_n
( 
\sum_{i=1}^d \hat{v}_i
\sum_{j=1}^d \hat{v}_j
\partial_{F_i}\partial_{F_j} f 
)
)\|_{L^2(U)}
\\
&\leq
d
\| b_n \|_{L^{\infty}(U)} \| b_m \|_{L^{\infty}(U)}
\| \hat{v} \|^2_{[L^{\infty}(U)]^d} \\
&\sum_{i = 1}^{d} \sum_{j = 1}^{d}
|\partial_{F_{i}} \partial_{F_{j}} f| \\
\end{split}
\label{errorbounds:eqn5}
\end{equation}
Now, 
\begin{equation}
\sum_{i = 1}^{d} \sum_{j = 1}^{d} |\partial_{F_{i}} \partial_{F_{j}}
f| = {\bf 1}^{T} \partial^{2} f {\bf 1},
\label{errorbounds:eqn6}
\end{equation}
where $\partial^{2} f$ refers to the Hessian of $f$. From the Chain
Rule for Hessians \cite{Scott2003} and adapting for Sobolev spaces
\cite{Adams1975} we obtain
\[
\partial^2 (f \circ F) = \partial F \partial^2 f \partial F^{T}
 + 
\nabla f \star \mcH,
\]
where $\partial^2 f$ refers to the Hessian of $f$,
\[
\nabla f \star \mcH := 
\left[
\begin{array}{ccc}
(\nabla f)^T \mcH_{1,1} & \dots & (\nabla f)^T \mcH_{1,d}  \\
\vdots & \ddots & \vdots  \\
(\nabla f)^T \mcH_{d,1} & \dots & (\nabla f)^T \mcH_{d,d}  \\
\end{array}
\right]
\]
and
\[
\mcH_{i,j} := 
\left[ 
\begin{array}{c}
\partial_{x_i} \partial_{x_j} F_1 \\
\vdots \\
\partial_{x_i} \partial_{x_j} F_d \\
\end{array}
\right]
\]
for all $i,j = 1,\dots,d$. It follows that
\begin{equation}
|{\bf 1}^T \partial^2 f \,{\bf 1}| 
\leq 
d \F^{-2}_{min}(
\| \partial^2 (f \circ F) \|_2
+
\| \nabla f \star \mcH \|_2
).
\label{errorbounds:eqn2}
\end{equation}
Furthermore,
\begin{equation}
\begin{split}
\| \nabla f \star \mcH\|_2 
&\leq
\| \nabla f \star \mcH\|_F 
\leq 
\sqrt{
\sum_{i,j=1}^d
| (\nabla f)^{T} \partial
F(\by)^{-1}\mcH_{i,j}|^2} \\
&\leq 
\| \nabla f \|_2
\F_{min}^{-1}
\sqrt{
\sum_{i,j=1}^d
\|\mcH_{i,j}\|^2_2} \\
&\leq 
d^{3/2}
\| \partial F^{-T} \nabla (f \circ F) \|_2
\F_{min}^{-1} 
 \sup_{i,j,n} |\partial_{x_i} \partial_{x_j} F_n| \\
&\leq 
d^{3/2}
\| (f \circ F) \|_{H^{1}(U)}
\F_{min}^{-2} 
 \sup_{i,j,n} |\partial_{x_i} \partial_{x_j} F_n| 
\end{split}
 \label{errorbounds:eqn3}
\end{equation}
Now, for $i,j = 1,\dots,N$ we have
\corb{\begin{equation}
|\partial_{x_i} \partial_{x_j} F_n| \leq 1 + 4 
\| \hat v \|_{[W^{1,\infty}(U)]^{d}}
\sum_{n = 1}^{N} \sqrt{\mu}_{n} \|b_{n}\|_{W^{2,\infty}(U)}.
\label{errorbounds:eqn4}
\end{equation}}
Combining \eqref{errorbounds:eqn5}, \eqref{errorbounds:eqn6},
\eqref{errorbounds:eqn2} , \eqref{errorbounds:eqn3} and
\eqref{errorbounds:eqn4} we obtain the result.
\end{proof}

\corb{From Lemmas \ref{perturbation:lemma5} and \ref{errorbounds:lemma1} -
\ref{errorbounds:lemma10} we have that for all $ \by_s \in \Gamma_s$
and for all $\byd_f \in \Gamma_f$
\[
|D^{2}_{\by_f} Q(\by)(\byd_f, \byd_f)|  
\leq \sum_{n,m=1}^{N_f} |\delta \tilde y^f_{n}| |\delta \tilde y^f_{m}| G_{n,m}, 
\]
where 
\[
\begin{split}
&G_{n,m}  
(\|(f
\circ F)(\bys,\bo)\|_{H^{2}(U)}, \| (\tilde{u} \circ F)(\bys,\bo)
\|_{H^1(U)}, C_P(U), \\ 
& 
\| \varphi(\bys,\bo) \|_{H^1(U)}, \| \bw(\bys,\bo) \|_{H^1(U)}
, \| B_{f,n} \|_2, \| B_{f,m} \|_2,
\|b_{f,m}\|_{W^{2,\infty}(U)}, \\
&\|b_{f,n}\|_{W^{2,\infty}(U)},
a_{max} 
\F_{max}, \F_{min},d, \| \hat{v} \|_{[W^{2,\infty}(U)]^d},
\sum_{l =
  1}^{N} \sqrt{\mu}_{s,l} \|b_{s,l}\|_{W^{2,\infty}(U)}
\\
&, \sum_{l =
  1}^{N} \sqrt{\mu}_{f,l} \|b_{f,l}\|_{W^{2,\infty}(U)}
)
\end{split}
\]
is a bounded constant that depends on the indicated
parameters. We have now proven the following result.}
\corb{\begin{theorem} 
  For all
  $\bys \in \Gamma_s$ and $\byf \in \Gamma_f$
\[
\| Q(\bys,\byf) - \hat{Q}(\bys,\byf)
\|_{L^2_{\rho}(\Gamma)} 
\leq \frac{1}{2} \sum_{n,m=1}^{N_f}
\sqrt{\mu_{f,n}} 
\sqrt{\mu_{f,m}} 
G_{n,m}
\leq \gset
(\sum_{k=1}^{N_f}
\sqrt{\mu_{f,k}})^2,
\]
where $\gset := \frac{1}{2} \sup_{n,m} G_{n,m}$.
\end{theorem}}
\subsection{Finite element error} 
\label{erroranalysis:finiteelement}
The finite element convergence rate is directly depend on the
regularity of the solution $u$ and influence function $\varphi$, the
polynomial order $H_{h}(U) \subset H^{1}_{0}(U)$ of the finite element
space and the mesh size $h$). By applying the triangular and Jensen
inequalities we obtain
\[
\corb{
\begin{split}
&\| 
\hat{Q}(\bys,\byf) - \hat{Q}_h(\bys,\byf) \|_{L^{1}_{\rho}(\Gamma)} 
\leq \|
Q(\bys,\bo) - Q_h(\bys,\bo) \|_{L^{1}_{\rho}(\Gamma_s)} \\ 
& + \sum_{n =
  1}^{N_f} \sqrt{\mu_{f,n}} \|\int_{U} \tilde \alpha_{n}(x,\by_s,\bo) -
\tilde \alpha_{i,h}(x,\by_s,\bo)\|_{L^{2}_{\rho}(\Gamma_s)}.
\end{split}
}
\]
Following a duality argument we obtain
\[
\corb{\|Q(\by_s,\bo) - Q_h(\by_s,\bo)\|_{L^{1}_{\rho}(\Gamma_s)} \leq 
a_{max} \F^{d}_{max}
\F^{-2}_{min} C_{\Gamma_s}(r)D_{\Gamma_s}(r) h^{2r}.}
\]
\corb{for some constant $r \in \N$, $C_{\Gamma_s}(r):= \int_{\Gamma_s}
  C(r,u(\bys,\bo))\rho(\by_s)d\by$ and $D_{\Gamma_s}(r):=
  \int_{\Gamma_s} C(r,\varphi(\bys,\bo))$ $\rho(\bys)d\by$.}  The
constant $r$ depends on the polynomial degree of the finite element
basis and the regularity properties of the solution $\tilde u \circ F$
(which is dependent on the regularity of $f$, the diffusion
coefficient $a$ and the mapping $F$).  It follows that
\corb{\begin{equation}
\| \hat{Q}(\bys,\bo) - \hat{Q}_h(\bys,\bo) \|_{L^{2}_{\rho}(\Gamma)} 
 \leq 
\sset_0 h^{2r}
+ 
 h^{r} \sum_{n = 1}^{N_f} \sset_n \sqrt{\mu_{f,n}} 
\label{errorestimates:finiteelement:eqn1}
\end{equation}}
where $\sset_0 :=a_{max} \F^{d}_{max} \F^{-2}_{min}
C_{\Gamma_s}(r)D_{\Gamma_s}(r)$ and
\[
\corb{
\begin{split}
&\sset_n
(\|(f
\circ F)(\bys,\bo)\|_{L^{2}(U)}, 
\| \bw(\bys, \bo) \|_{H^1(U)}
, \| B_{f,n} \|_2, \|b_{f,n}\|_{L^{\infty}(U)}, \\
&a_{max} 
\F_{max}, \F_{min},d, \| \hat{v} \|_{[L^{\infty}(U)]^d},
C_{\Gamma_s}(r),D_{\Gamma_s}(r)
)
\end{split}
}
\]
are bounded constants for $n = 1, \dots, N_f$.

\subsection{Sparse grid error} 
For the sake of simplicity, we will only explicitly show the
convergence rates for the isotropic Smolyak sparse grid. However, this
analysis can be extended to the anisotropic case without much
difficulty.  Now, we have that
\[
\corb{
\begin{split}
\|\hat{Q}_{h}(\bys,\byf) - \mcS_{w}^{m,g} \hat{Q}_{h}(\bys,\byf) 
\|_{L^2_{\rho}(\Gamma)}
& \leq
a_{max}\F^{d}_{max} \F_{min}^{-2}
   \|e_{0}\|_{L^{2}_{\rho}(\Gamma_{s};H^{1}_{0}(U))} \\
&+
\sum_{n = 1}^{N_f} \sqrt{\mu_{f,i}} 
\|e_{n}\|_{L^{2}_{\rho}(\Gamma_{s})},
\end{split}}
\]
\corb{where $e_0:=\hat{u}_h(\bys,\bo) -
  \mcS^{m,g}_w[\hat{u}_h(\bys,\bo)]$ and
\[
e_n: = \int_U \tilde
  \alpha_{n,h}(\bys,\bo) - \mcS^{m,g}_w[ \int_U \tilde
    \alpha_{n,h}(\bys,\bo)] 
\]
for $n = 1,\dots,N_f$, and
\[
L^{q}_{\rho}(\Gamma_s;V) := \{v:\Gamma_s \times U \rightarrow V\,\, 
\mbox{is strongly
  measurable,} \,\, \int_{\Gamma} \|v\|^{q}_{V}\,\rho(\by)\,d\by <
\infty \}.
\]}
for any Banach space $V$ defined on $U$.

In \cite{nobile2008b,nobile2008a} the error estimates for isotropic
and anisotropic Smolyak sparse grids with Clenshaw-Curtis and Gaussian
abscissas are derived. It is shown that
$\|e_0\|_{L^{2}_{\rho}(\Gamma_{s};H^{1}_{0}(U))}$ (and
$\|e_n\|_{L^{2}_{\rho}(\Gamma_{s})}$ for $n = 1, \dots, N_f$) exhibit
algebraic or sub-exponential convergence with respect to the number of
collocation knots $\eta$. \corb{For these estimates to be valid it is
  assumed that the semi-discrete solution $\hat u_{0,h} :=
  \hat{u}_{h}({\by_s,\bo})$ and $\hat u_{n,h} := \int_U \tilde
  \alpha_{k,n}(\cdot,\by_s,\bo)$, $n = 1,\dots,N_f$ admit an analytic
  extension in the same region $\Theta_{\beta,N_s}$. This is a
  reasonable assumption to make.}

\corb{Consider the polyellipse in $\mcE_{\sigma_1, \dots, \sigma_{N_s}}
:= \Pi_{n=1}^{N_s}\mcE_{n,\sigma_n} \subset \C^{N_s}$ where
\[
\begin{split}
  \mcE_{n,\sigma_n} &:= \left\{ z \in \C; \sigma_n > 0;
  \sigma_n \geq \kappa_n \geq 0;
  \,\Real(z) = \frac{e^{\kappa_n}
 +
 e^{-\kappa_n}}{2}cos(\theta), \right.\\
&\left.
  \Imag(z)= \frac{e^{\kappa_n} -
  e^{-\kappa_n}}{2}sin(\theta), \theta \in [0,2\pi) \right\},\,\,
\end{split}
\]
and
\[
\Sigma_{n} := \left\{ z_n \in \mathbb{C};\, y_n = y
+ w_n,\,y \in [-1,1],\, |w_n| \leq \tau_n :=
\frac{\beta}{1 - \tilde{\delta}}
\right\}
\]
for $n = 1,\dots,N_s$.  For the sparse grid error estimates to be
valid the solution $(\tilde{u}_h(\cdot,\bys,\bo)$ and $\int_U \tilde
\alpha_{n,h}(\cdot,\by_s,\bo)$, $n = 1,\dots,N_f$, have to admit an
extension on the polyellipse $\mcE_{\sigma_1, \dots,
  \sigma_{N_s}}$. The coefficients $\sigma_n$, for $n = 1, \dots, N$
control the overall decay $\hat \sigma$ of the sparse grid error
estimate. Since we restrict our attention to isotropic sparse grids
the decay will be dictated by the smallest $\sigma_n$ i.e.
$\hat{\sigma} \equiv \min_{n = 1,\dots, N_s} \sigma_n.$}

\corb{The next step is to find a suitable embedding of
  $\mcE_{\sigma_1, \dots, \sigma_{N_s}}$ in $\Theta_{\beta,
    N_s}$. Thus we need to pick the largest $\sigma_n$, $n =
  1,\dots,N_s$ such that $\mcE_{\sigma_1, \dots, \sigma_{N_s}} \subset
  \Theta_{\beta, N_s}$. This is achieved by forming the set
  $\Sigma:=\Sigma_1 \times \dots \times \Sigma_{N_s}$ and letting
  $\sigma_1 = \sigma_2 = \dots = \sigma_{N_s} = \hat{\sigma} =
  \log{(\sqrt{\tau^2_{N_s} + 1} + \tau_{N_s} )} > 0$ as shown in
  Figure \ref{errorbounds:fig1}.}

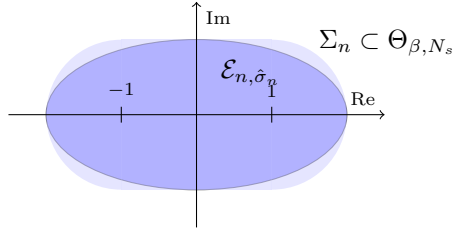
\begin{figure}[h]
\begin{center}
\begin{tikzpicture}
    \begin{scope}[font=\scriptsize]
    \filldraw[fill=blue!80!,semitransparent] (0,0) ellipse (2 and 1);

    \node[shape=semicircle, rotate=270,fill=blue!20,semitransparent,inner
      sep=12.7pt, anchor=south, outer sep=0pt] at (1,0) (char) {};
    \node[shape=semicircle, rotate=90,fill=blue!20,semitransparent,inner
      sep=12.7pt, anchor=south, outer sep=0pt] at (-1,0) (char) {};
    \path [draw=none,fill=blue!20,semitransparent] (-1.001,-1) rectangle
    (1.001,1.001);
    
    \draw [->] (-2.5, 0) -- (2.5, 0) node [above left]  {$\Real $};
    \draw [->] (0,-1.5) -- (0,1.5) node [below right] {$\Imag$};
    \draw (1,-3pt) -- (1,3pt)   node [above] {$1$};
    \draw (-1,-3pt) -- (-1,3pt) node [above] {$-1$};
    \end{scope}
    
    \node [below right] at (1.50,1.25) {$\Sigma_n \subset \Theta_{\beta,N_s}$};
    \node [below right] at (0.2,0.85) {$\mcE_{n,\hat \sigma_n}$}; 
\end{tikzpicture}
\end{center}
\caption{\corb{Embedding of $\mcE_{n,\hat \sigma_n}$ in $\Sigma_n
  \subset \Theta_{\beta, N_s}$.}}
\label{errorbounds:fig1}
\end{figure}

\corb{We now have almost everything we need to state the sparse grid
  error estimates. However, in \cite{nobile2008a} to simplify the
  estimate it is assumed that if $v \in C^{0}(\Gamma;H^{1}_{0}(U))$
  then the term $M(v)$ (see page 2322) is equal to one. We reintroduce
  the term $M(v)$ and note that it can be bounded by $\max_{\bz \in
    \Theta_{\beta,N_s}}$ $\|v(\bz)\|_{H^{1}_{0}(U)}$ and update the
  sparse grids error estimate. To this end let $\tilde M : = \max_{n =
    0}^{N_s}\max_{\bz \in \Theta_{N_s,\beta}} \| \tilde u_{n,h}(\bz)
  \|_{H^{1}_{0}(U)}$.
\begin{remark} In \cite{Castrillon2013} Corollary 8 a bound for
  $\| \hat u(\cdot, \bz) \|_{H^{1}_{0}(U)}$, $\bz \in
  \Theta_{\beta,N_s}$, can be obtained by applying the Poincar\'{e}
  inequality. Following a similar argument a bound for $\| \hat
  \varphi(\cdot, \bz) \|_{H^{1}_{0}(U)}$ for all $\bz \in
  \Theta_{\beta,N_s}$. Thus bounds for $\| \tilde u_{n,h}(\bz)
  \|_{H^{1}_{0}(U)}$ for $n = 0,\dots, N_s$ and for all $\bz \in
  \Theta_{\beta,N_s}$ can be obtained.
\end{remark}
}

\corb{Modifying Theorem 3.11 in \cite{nobile2008a} it can be shown
  that given a sufficiently large $\eta$ ($w > N_s / \log{2}$) a
  Smolyak sparse grid with a nested Clenshaw Curtis abscissas we
  obtain the following estimate
\begin{equation}
\|e_n\|_{L^{2}_{\rho} (\Gamma_{s})} \leq
\mcQ(\sigma,\delta^{*}(\sigma),N_s,\tilde
M)\eta^{\mu_3(\sigma,\delta^{*},N_s)}\exp \left(-\frac{N_s
  \sigma}{2^{1/N_s}} \eta^{\mu_2(N_s)} \right) \\
\label{erroranalysis:sparsegrid:estimate2}
\end{equation}
for $n = 0, \dots, N_f$, where $\sigma = \hat{\sigma}/2$,
$\delta^{*}(\sigma) := (e \log{(2)} - 1) / \tilde C_{2}(\sigma)$,
\[\mcQ(\sigma,\delta^{*}(\sigma),N_s, \tilde M) := 
\frac{ C_1(\sigma,\delta^{*}(\sigma),\tilde M)}{\exp(\sigma \delta^{*}(\sigma)
  \tilde{C}_2(\sigma) )}
\frac{\max\{1,C_1(\sigma,\delta^{*}(\sigma),\tilde M)\}^{N_s}}{|1 -
  C_1(\sigma,\delta^{*}(\sigma),\tilde M)|},
\]
$\mu_2(N_s) = \frac{log(2)}{N_s(1 + log(2N_s))}$ and
$\mu_3(\sigma,\delta^{*}(\sigma),Ns) = \frac{\sigma
  \delta^{*}(\sigma)\tilde{C}_2(\sigma)}{1 + \log{(2N_s)}}$.
Furthermore, $C(\sigma) = \frac{4}{e^{2\sigma} - 1}$,
\[
\begin{split}
\tilde{C}_{2}(\sigma) &= 1 + \frac{1}{\log{2}}\sqrt{
\frac{\pi}{2\sigma}}
,\,\,\delta^{*}(\sigma) = \frac{e\log{(2)} - 1}{\tilde{C}_2
  (\sigma)},\\
C_1(\sigma,\delta,\tilde M)
&=
\frac{4\tilde{M}
C(\sigma)a(\delta,\sigma)}{
  e\delta\sigma},
\end{split}
\]
and
\[
a(\delta,\sigma):=
\exp{
\left(
\delta \sigma \left\{
\frac{1}{\sigma \log^{2}{(2)}}
+ \frac{1}{\log{(2)}\sqrt{2 \sigma}}
+ 2\left( 1 + \frac{1}{\log{(2)}} 
\sqrt{ \frac{\pi}{2\sigma} }
\right)
\right\}
\right)
}.
\]
}


\section{Complexity and tolerance}
\label{complexity} 
\corb{In this section we derive the total work $W$ needed such that
  $|var[Q(\bys,$ $\byf)] - var[\mcS^{m,g}_w[\hat{Q}_{h}(\bys,\byf)]]
  |$ and $|\mathbb{E}[Q(\bys,\byf)]$ $-\mathbb{E}[\mcS^{m,g}_w$ $[
      Q_{h}(\bys,\byf)]]|$ for the isotropic CC sparse grid is less or
  equal to a given tolerance parameter $tol \in \R^{+}$.}

Let $N_{h}$ be the number of degrees of freedom used to compute the
semi-discrete approximation $u_h \in H_{h}(U) \subset H^{1}_{0}(U)$.
We assume that the computational complexity for solving $u_{h}$ is
$\mcO( N^{q}_{h})$ for each realization, where the constant $q \geq 1$
reflects the optimality of the finite element solver. The cost for
solving the approximation of the influence function $\varphi_{h} \in
H_{h}(U)$ is also $\mcO( N^{q}_{h})$. \corb{Thus for any $\bys \in
  \Gamma_{s}$, the cost for computing
  $Q_{h}(\bys,\bo):=B(\bys,\bo;u_{h}(\bys,\bo),
  \varphi_{h}(\bys,\bo))$ is bounded by $\mcO(N_hd^2 + N^{q}_{h})$.
  Similarly, for any $\bys \in \Gamma_{s}$ the cost for evaluating
  $\int_U \tilde \alpha_{n,h}(\cdot, \bys, \bo)$ is $\mcO(N_{h}d^2 +
  N^q_h)$.}

\corb{
  \begin{remark}
To compute the expectation integrals for the mean and variance
correction a Gauss quadrature scheme coupled with an {\it auxiliary}
probability distribution $\hat \rho(\by)$ such that
  \[
  \hat \rho(\by) = \Pi_{n = 1}^{N} \rho_n(y_n)\,\,\,
  \mbox{and $\rho / \hat \rho < C < \infty$.}
  \]
  for some $C > 0$ (See \cite{Castrillon2013} for details). However,
  to simplify the analysis it is assumed that quadrature is exact and
  of cost $\mcO(1)$.
\end{remark}
}
\corb{
  Let $\mcS^{m,g}_{\lv}$ be the sparse grid operator characterized by
  $m(i)$ and $g({\bf i})$. Furthermore, let
  $\eta_0(N_s,m,g,\lv,\Theta_{\beta,N_s})$ be the number of the sparse
  grid knots for constructing $\mcS^{m,g}_{\lv} [\tilde
    \alpha_{n,h}(\cdot,\bys,\bo)]$ and
  $\eta_n(N_s,m,g,\lv,\Theta_{\beta,N_s})$ for constructing
  $\mcS^{m,g}_{\lv} [\tilde \alpha_{n,h}(\cdot,\bys,\bo)]$, for $n =
  1,\dots,N_f$. The cost for computing $\mathbb{E}[\mcS^{m,g}_{\lv} $
    $[Q_{h}(\bys,\bo )]]$ is $\mcO( (N_hd^2 + N^{q}_{h})\eta_0)$
  and the cost for computing $\sum_{n = 1}^{N_f}\sqrt{\mu_{f,n}}$ $
  \esets{ \tilde y^{f}_{n} \mcS^{m,g}_{\lv}[\int_U \tilde
      \alpha_{n}(\cdot,\bys,\bo)]}$ is bounded by $\mcO( (N_hd^2 +
  N^{q}_{h})N_f\eta)$, where
  \[
  \eta := \max_{n=0,\dots,N_f} \eta_n.
  \]
  The total cost for computing the mean correction is bounded by
\begin{equation}
W^{mean}_{Total}(tol) = \mcO( (N_h(tol)d^2 + N^{q}_{h}(tol))N_f(tol)\eta(tol)).
\label{complexity:eqn1}
\end{equation}
Following a similar argument the cost for computing the
variance correction is bounded by
\begin{equation}
W^{var}_{Total}(tol) = \mcO( (N_h(tol)d^2 + N^{q}_{h}(tol))N^2_f(tol)\eta(tol)).
\label{complexity:eqn2}
\end{equation}
} We now obtain the estimates for $N_h(tol)$, $N_f(tol)$ and
$\eta(tol)$ for the Perturbation, Finite Element and Sparse Grids
respectively:

(a) {\bf Perturbation:} From the truncation estimate derived in
Section \ref{erroranalysis:perturbation} we seek \corb{$\|Q(\by_s,$
  $\by_f) - \hat{Q}(\by_s,\by_f)\|_{L^{2}_{\rho}(\Gamma)} \leq
  \frac{tol}{3C_{P}}$} with respect to the decay of the coefficients
$\sqrt{\mu_{f,n}}$, $n = 1, \dots N_f$.  First, make the assumption
that $B_T := \sum_{n = 1}^{N_f} \sqrt{\mu_{f,n}} \leq C_{D}N^{-l}_s$
for some uniformly bounded $C_{D}>0$ and $l > 0$.  It follows that
\corb{$\|Q(\by_s,$ $\by_f) -
  \hat{Q}(\by_s,\by_f)\|_{L^{2}_{\rho}(\Gamma)} \leq
  \frac{tol}{3C_{P}}$} if
\[
B_T^2 \gset \leq C^2_{D}N^{-2l}_s \gset \leq \frac{tol}{3 C_P}.
\]
Finally, we have that
\[
N_{f}(tol) \geq \left( \frac{tol}{3C_PC_D^2\gset}
\right)^{-1/(2l)}.
\]

(b) {\bf Finite Element:} From Section
  \ref{erroranalysis:finiteelement} if
\[
\sset_0 h^{2r} + B_T \T_0 h^r \leq \frac{tol}{3 C_{PFE}}, 
\]
\corb{$\T_0 := \max_{n = 1}^{N_f} \sset_n$, then
  $\|\hat{Q}(\bys,\bo) - \hat{Q}_h(\bys,\bo) \|_{L^{2}_{\rho}(\Gamma;
    H^{1}_{0}(U))} \leq \frac{tol}{3C_{PFE}}$.} Solving the quadratic
inequality we obtain that
\[
h(tol) \leq 
\left(
-\frac{B_T \T_0}{2 \sset_0}
+
\left(
\left(
\frac{B_T \T_0}{4\sset_0}
\right)^2
+
\frac{4 tol}{12 \sset_0 C_{PFE}}
\right)^{1/2}
\right)
^{1/r}
\]
Assuming that $N_h$ grows as $\mcO(h^{-d})$ then
\[
N_h(tol) \geq 
D_3 \left(
-\frac{B_T \T_0}{2 \sset_0}
+
\left(
\left(
\frac{B_T \T_0}{4\sset_0}
\right)^2
+
\frac{4 tol}{12 \sset_0 C_{FE}}
\right)^{1/2}
\right)
^{-d/r}
\]
for some constant $D_3 > 0$.

(c) {\bf Sparse Grid:} \corb{We seek $\|\hat{Q}_{h}(\bys,\bo) -
  \mcS_{w}^{m,g} \hat{Q}_{h}(\bys,\bo) \|_{L^2_{\rho}(\Gamma)} \leq
  \frac{tol}{3C_{PSG}}$. This is satisfied if
  $\|e_0\|_{L^{2}_{\rho}(\Gamma_{s};H^{1}_{0}(U))} \leq
  \frac{tol}{6a_{max}\F^d_{max}\F^{-2}_{min}C_{PSG}}$
and}
  \[
  \corb{
  \| e_n \|_{L^{2}_{\rho}(\Gamma_{s};L^2(U))} \leq
  \frac{tol}{6
    B_T C_{PSG}} }
  \]
for $n = 1, \dots, N_f$.  Following the same strategy as in
\cite{nobile2008a} (equation (3.39)), to simplify the bound
\eqref{erroranalysis:sparsegrid:estimate2} choose $\delta^{*} = (e
\log{(2)} - 1)/\tilde{C}_2(\sigma)$.  Thus $\|\hat{Q}_{h}(\by) -
\mcS_{w}^{m,g} \hat{Q}_{h}(\by) \|_{L^2_{\rho}(\Gamma)} \leq
\frac{tol}{3C_{PSG}}$ if
\[
\corb{\eta_0(tol) \geq 
\left( \frac{
 6 
a_{max}\F^d_{max}\F^{-2}_{min}C_{PSG} C_{F}F^{N_s}
  \exp(\sigma(\beta) )}{tol 
}
\right)^{\frac{1 + \log(2N_s)}{\sigma}} 
}
\]
for a sufficiently large $N_s$, where \corb{$C_F :=
  \frac{C_1(\sigma,\delta^{*},\tilde M)}{|1 - C_1(\sigma,\delta^{*},
    \tilde M)|}$, and $F: = \max\{1,$ $C_1(\sigma,\delta^{*},\tilde
  M)\}$}. Similarly, for a sufficiently large $N_s$ we have that
\[
\corb{
\eta_n(tol) \geq 
C
\left( \frac{
6 
B_TC_{PSG} C_{F}F^{N_s}
  \exp(\sigma(\beta) )}{tol
}
\right)^{\frac{1 + \log(2N_s)}{\sigma}} 
}
\]
for $n = 1, \dots, N_f$.

\bigskip
\corb{Combining (a), (b) and (c) into equations
  \eqref{complexity:eqn1} and \eqref{complexity:eqn2} we obtain the
  total work $W^{mean}_{Total}(tol)$ and $W^{var}_{Total}(tol)$ as a
  function of a given user error tolerance $tol$.}

\section{Numerical results}
\label{numericalresults}
\corb{In this section we test the hybrid collocation-perturbation
  method on an elliptic PDE with stochastic deformation of the unit
  square domain i.e. $U = (0,1) \times (0,1)$. The deformation map
  $F:U \rightarrow \mcD(\omega)$ is given by
\[
\begin{array}{llll}
  F(x_{1}, x_{2}) = (x_{1},\,(x_{2}-0.5)(e(x_{1},\omega)) + 0.5) 
  & & if & x_{2} > 0.5\\
  F(x_{1}, x_{2}) = (x_{1},\,x_{2}) & & if & 0 \leq x_{2} \leq 0.5.
\end{array}
\]
According to this map only the upper half of the square is deformed
but the lower half is left unchanged.  The cartoon example of the
deformation on the unit square $U$ is shown in Figure
\ref{numericalresults:fig0}.}
\begin{figure}[h]
\begin{center}
\begin{tikzpicture}[scale=1.5] 
    \begin{scope}[thick,font=\scriptsize]
      \path [draw=none,fill=blue!80,semitransparent] (0,0) 
      rectangle (1.001,1.001);
      \draw [->] (0,0) -- (1.5,0) node [above left]  {$x_1$};
      \draw [->] (0,0) -- (0,1.5) node [below right] {$x_2$};
      \draw (1,-3pt) -- (1,3pt)   node [right] {$1$};
    \end{scope}
    \draw[thin] (0,0) -- (1,0) -- (1,1) -- (0,1) -- (0,0);
    \draw[thin,dashed] (0,0.5) -- (1,0.5);
    \node [below right,black] at (0.5,1.5) {$U$};
    \coordinate (O) at (2,0.5);
    \coordinate (P) at (3.5,0.5);
    \draw[->, >=latex, gray, line width=4 pt] (O) -- (P);
    \node [below right,black] at (2.3,1) {$F(x_1,x_2)$};
    \node [below right,black] at (0.35,0.45) {$\tilde D$};
\end{tikzpicture}
\hspace{10mm}
\begin{tikzpicture}[scale=1.5] 
    \begin{scope}[thick,font=\scriptsize]
      \path [draw=none,fill=blue!80,semitransparent] (0,0) 
    rectangle (0.5,1);  
       \path [draw=none,fill=blue!80,semitransparent] (0.5,0) 
    rectangle (1,0.99);  
       \draw [->] (0,0) -- (1.5,0) node [above left]  {$y_1$};
       \draw [->] (0,0) -- (0,1.5) node [below right] {$y_2$};
       \draw (1,-3pt) -- (1,3pt)   node [right] {$1$};
       \draw[thin,dashed] (0,0.5) -- (1,0.5);
    \end{scope}
    \draw[fill = blue!80,semitransparent] plot[ultra thin,domain=0:0.5] (\x,{1 + 0.25 * sin(2 *
      pi * \x r)});
    \draw[fill = white] plot[ultra thin,domain=0.5:1] (\x,{1 + 0.25 * sin(2 *
      pi * \x r)});
    \draw [thin] (1,0) -- (1,1);
    \node [below right,black] at (0.5,1.5) {$\mcD(\omega)$};
    \node [below right,black] at (0.35,0.45) {$\tilde D$};
\end{tikzpicture}
\end{center}
\caption{\corb{Stochastic deformation of unit square $U$ according to
    the rule given by $F:U \rightarrow \mcD(\omega)$. The region
    $\tilde D$ is not deformed and given by $(0,1) \times (0,0.5)$.}}
\label{numericalresults:fig0}
\end{figure}
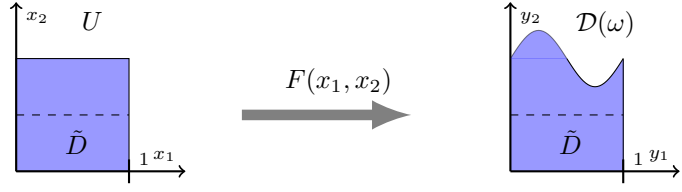

\corb{The Dirichlet boundary conditions are set according to the following
rule:
\[
u(x_{1},x_{2})|_{\partial D(\omega)}
= \Big\{
\begin{array}{ll}
  \vartheta(x_{1}) & \mbox{upper border} \\
  0       &  \mbox{otherwise}
\end{array}
\]
where $\vartheta(x_{1}) : =\exp( \frac{-1}{1 - 4(x_{1}-0.5)^2})$. Note that
the boundary condition on the upper border does not change even after
the stochastic perturbation.}

\corb{For the stochastic model $e(x_1,\omega)$ we use a variant of the
  Karhunen Lo\`{e}ve expansion of an exponential oscillating kernel
  that are encountered in optical problems \cite{Kober2003}. This
  model is given by
\begin{center}
$e_{s}(\omega,x_{1}) := 1 + cY_{1}(\omega)\left(
    \frac{\sqrt{\pi}L}{2} \right)^{1/2} + c \sum_{n = 2}^{N_{s}} \sqrt{\mu_{n}}
  \varphi_{n}(x_{1})Y_{n}(\omega); \hspace{1mm}$ $e_{f}(\omega,x_{1})
  := c \sum_{n = 1}^{N_f} \sqrt{\mu_{n+N_s}}
  \varphi_{n}(x_{1})Y_{n}(\omega)$
\end{center}
with decay $\sqrt{\mu_{n}} := \frac{(\sqrt{\pi}L)^{1/2}}{n^{k}}$, $n
\in \N$, $k \in \R^{+}$ and
\[
\varphi_{n}(x_{1}) : = 
 \frac{
sin 
\left( \frac{n \pi x_{1}}{2 L_{p}} \right) 
-
cos
\left( \frac{n \pi x_{1}}{2 L_{p}} \right) 
+
cosh(x_1) + sinh(x_1)}{n}.
 \]
It is assumed that $\{Y_{n}\}_{n = 1}^{N}$ are independent uniform
distributed in $(-\sqrt{3},\sqrt{3})$, thus $\eset{Y_n} = 0$,
$\eset{Y_nY_m} = \delta[n-m]$ for $n,m = 1 \dots N$ where
$\delta[\cdot]$ is the Kronecker delta function.}

\corb{It can be shown that for $n > 1$ we have that
\[
B_n = \left[ \begin{array}{cc}
0 & 0\\
c(x_2 - 0.5) \partial_{x_1} \varphi_n(x_1) & 0 
\end{array}
\right].
\]
This implies that $\sup_{x \in U} \sigma_{max}(B_{l}(x))$ is bounded
by a constant. Thus for $k = 1$ we obtain linear decay on the gradient
of the deformation.  In Figure \ref{numericalresults:fig1} two mesh
examples of the domain $U$ and a particular realization of
${\mcD(\omega)}$ with the model $e(x_1,\omega)$ are shown with the
Dirichlet boundary conditions.}

\corb{The QoI is defined on the bottom half of the reference domain
  ($\tilde D$), which is not deformed, as
\[
Q(\hat u) := \int_{(0,1)} \int_{(0,1/2)}
\vartheta(x_{1})\vartheta(2x_{2})\hat u(\omega,x_1,x_2)
\,dx_{1}dx_{2}.
\]
In addition, we have the following:
}
\begin{figure}[htp]
\hspace{-7mm}
\psfrag{a}[cb]{\tiny Nominal Domain}
\psfrag{x}[c]{\tiny $x_{1}$} \psfrag{$x_{2}$}[c]{\tiny y} 
\psfrag{quantityofinterest}[][r]{\tiny QoI } 
\psfrag{stochasticdomain}[][r]{\tiny St. Domain}
\psfrag{Realization}[br][br]{\small Realization}
\psfrag{Reference}[br][br]{\small Reference}
\psfrag{d1}[c][bl][0.8][90]{\small $\hat u|_{\partial U} = 0$} 
\psfrag{d2}[c][br][0.8][0]{\small $\hat u|_{\partial U} = 0$} 
\psfrag{d3}[c][bl][0.8][90]{\small $\hat u|_{\partial U} = 0$} 
\psfrag{d4}[c][br][0.8][0]{\small $\hat u|_{\partial U} = 
\vartheta(x_1)$}

   \includegraphics[ width=5.0in, height=2.1in
  ]{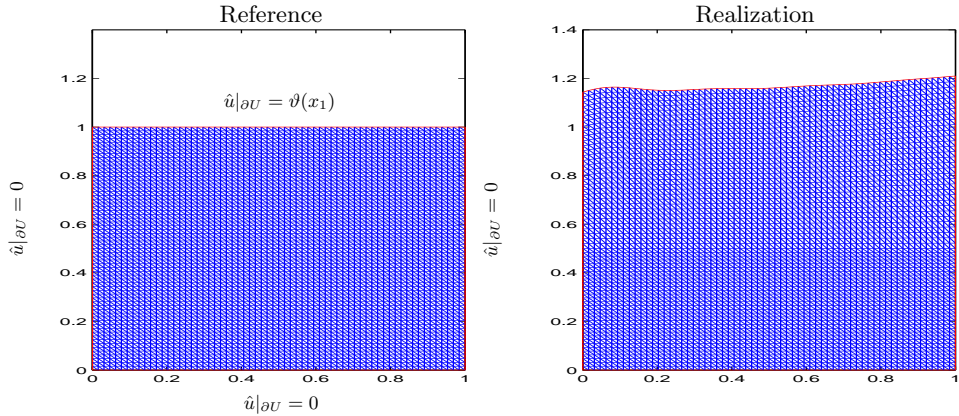} 
  \caption{Stochastic deformation of a square domain. (left) Reference
    square domain with Dirichlet boundary conditions. (right) Vertical
    deformation from stochastic model.}
\label{numericalresults:fig1}
\end{figure}

\corb{
\begin{enumerate}[(i)]
\item $a(x) = 1$ for all $x \in U$, $L = 1/2$, $L_P = 1$, $N = 15$.
  \item The domain is discretized with a $2049 \times 2049$ triangular
    mesh.
\item 
$\mathbb{E}[Q_{h}]$, $\mathbb{E}[Q^2_{h}]$, and $\sum_{i = 1}^{N_f}
  \mu_{f,i} \esets{ \int_U \tilde \alpha_{i,h}}^2$ are computed with
  the Clenshaw-Curtis isotropic sparse grid from the {\it Sparse Grids
    Matlab Kit} \cite{Tamellini2015,Back2011}.
\item The reference solutions $\var[Q_{h}(u_{ref})]$ and
  $\mathbb{E}[Q_{h}(u_{ref})]$ are computed with a dimension adaptive
  sparse grid ({\it Sparse Grid Toolbox V5.1}
  \cite{Gerstner2003,spalg,spdoc}) with Chebyshev-Gauss-Lobatto
  abscissas for $N = 15$ dimensions.
\item The QoI is normalized by the reference solution $Q(U)$.
\item The reference computed mean value is 1.054 and variance is
  0.1122 (0.3349 std) for $c = 1/15$ and cubic decay ($k = 3$).
\end{enumerate}
}

\corb{
\begin{remark}
The correction variance term is computed on the fixed reference domain
$U$ as described by Problem \ref{setup:Prob3} instead of the perturbed
domain. The pure collocation approach (without the variance
correction) and reference solution are also computed on $U$. Numerical
experiments confirm that computing the pure collocation approach on
$U$, as described by Problem \ref{setup:Prob3}, or the perturbed
domain $\mcD(\omega)$ lead to the same answer up to the finite
element error. This is consistent with the theory.
\end{remark}
}

\corb{For the first numerical example we assume that we have cubic
  decay of the deformation i.e. the gradient terms $\sqrt{\mu_{n}}
  \sup_{x \in U} \|B_n(x)\|$ decay as $n^{-3}$. The domain is formed
  from a $2049 \times 2049$ triangular mesh. The reference domain is
  computed with 30,000 knots (dimension adaptive sparse grid).  In
  Figure \ref{results:fig2}(a) we show the results for the hybrid
  collocation-perturbation method for $c = 1/15$, $k = 3$ (cubic
  decay), $N_s = 2,3,4$ dimensions and compare them to the reference
  solution. For the collocation method the level of accuracy is set to
  $w = 5$. For the variance correction we use $w = 3$ since the there
  is no benefit to increase $w$ as the sparse grid error is smaller
  than the perturbation error. The observed computational cost for
  computing the variance correction is about $10\%$ of the collocation
  method.}

  \corb{In Figure \ref{results:fig2}(b) we compare the results between
    the pure collocation \cite{Castrillon2013} and hybrid
    collocation-perturbation method. Notice the hybrid
    collocation-perturbation shows a marked improvement in accuracy
    over the pure collocation approach.}

\begin{figure}[htb]
\begin{center}
\begin{tabular}{cc}
\psfrag{AAAAAAAAA1}[][][0.8]{\tiny $N_s = 3$}
\psfrag{AAAAAAAAA2}[][][0.8]{\tiny $N_s = 4$}
\psfrag{AAAAAAAAA3}[][][0.8]{\tiny $N_s = 5$}
\psfrag{knots}{\tiny knots}
\psfrag{Var Error}[][c]{\tiny Var Error (Perturbation)}
\psfrag{b}[][]{\tiny $|Var[Q(u_{ref})] - Var[\mcS^{m,g}_w[\hat Q_h(\bys)]]|$} 
\includegraphics[width=2.2in,height=2.0in]{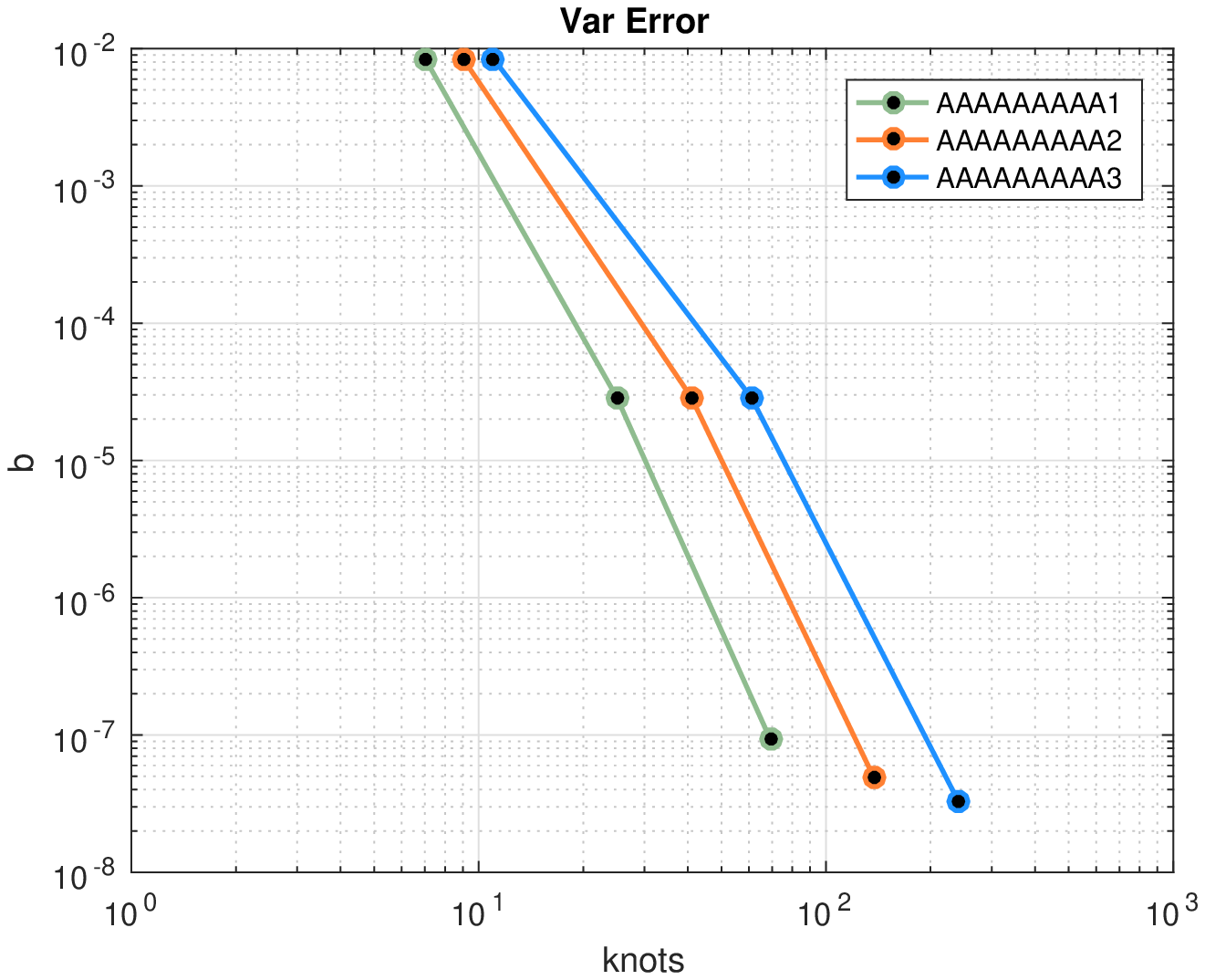}
&
\psfrag{knots}{\tiny knots}
\psfrag{b}[][c]{\tiny Var Error} 
\psfrag{Var Error}[][c]{\tiny Var Error (Collocation + Pert.)}
\psfrag{knots}[][c]{\tiny knots}
\psfrag{AAAAAAAAAAAAAAA1}[][][0.8]{\tiny Col ,$N_s = 3$}
\psfrag{AAAAAAAAAAAAAAA2}[][][0.8]{\tiny Pert,$N_s = 3$}
\psfrag{AAAAAAAAAAAAAAA3}[][][0.8]{\tiny Col ,$N_s = 4$}
\psfrag{AAAAAAAAAAAAAAA4}[][][0.8]{\tiny Pert,$N_s = 4$}
\psfrag{AAAAAAAAAAAAAAA5}[][][0.8]{\tiny Col ,$N_s = 5$}
\psfrag{AAAAAAAAAAAAAAA6}[][][0.8]{\tiny Pert,$N_s = 5$}
\includegraphics[width=2.2in,height=2.0in]{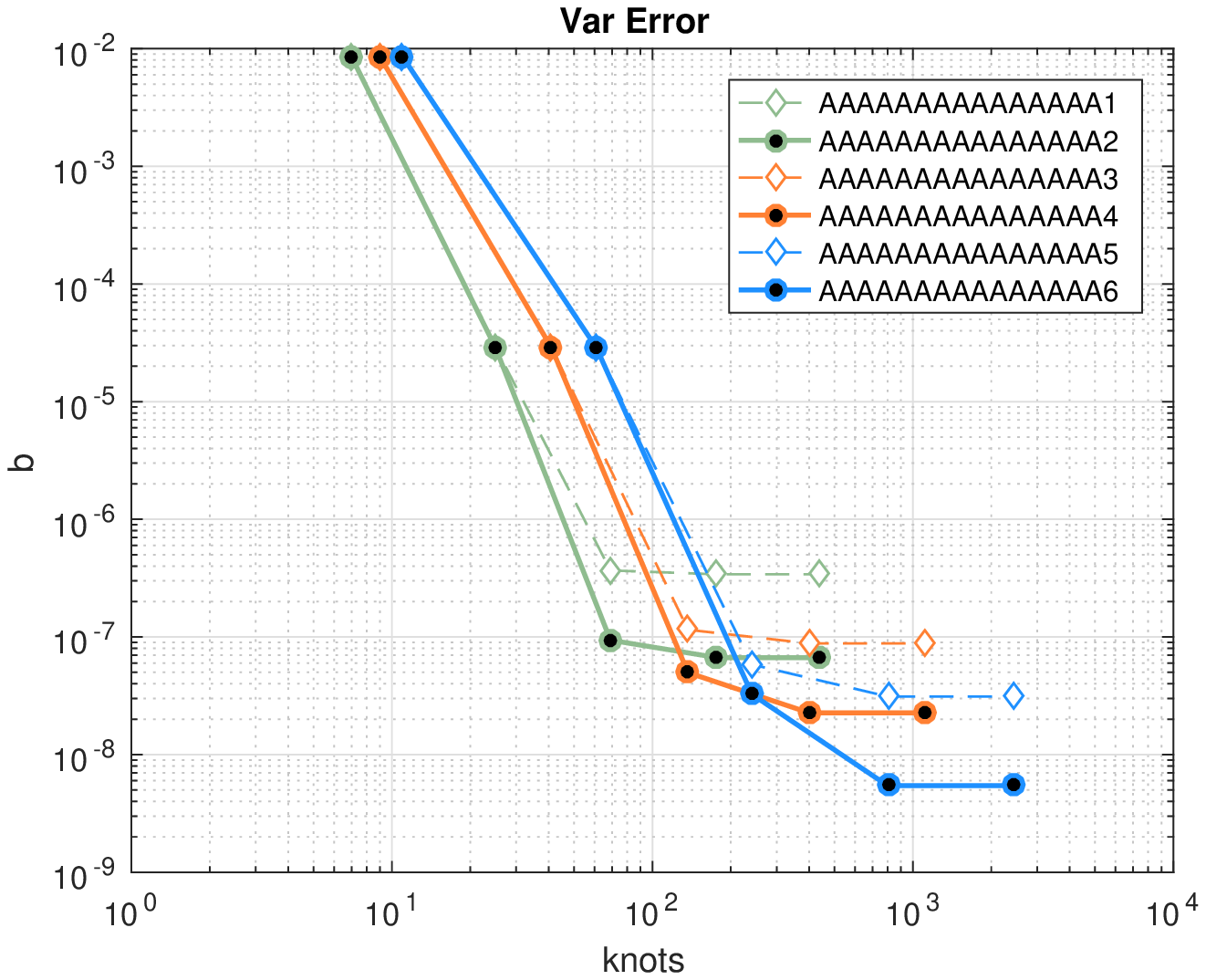} 
\\
(a) & (b)
\end{tabular}
\begin{tabular}{cc}
\psfrag{AAAAAAAAA1}[][][0.8]{\tiny $N_s = 2$}
\psfrag{AAAAAAAAA2}[][][0.8]{\tiny $N_s = 3$}
\psfrag{AAAAAAAAA3}[][][0.8]{\tiny $N_s = 4$}
\psfrag{knots}{\tiny knots}
\psfrag{Var Correction}[][c]{\tiny Var Error}
\psfrag{b}[][]{\tiny $|Var[Q(u_{ref})] - Var[\mcS^{m,g}_w[\hat Q_h(\bys)]]|$} 
\psfrag{b}[][]{} 
\\
\end{tabular}
\end{center}
\caption{\corb{Hybrid Collocation-Perturbation results with $k = 3$
    (cubic decay) and $c = 1/15$. (a) Variance error for the hybrid
    collocation-perturbation method with respect to the number of
    collocation samples with an isotropic sparse grid. The maximum
    level is set to $w=3$.  (b) Comparison between the pure
    collocation (Col) and the hybrid collocation-perturbation (Pert)
    approaches. As we observe the error decays significantly with the
    addition of the variance correction. However, the graphs saturate
    once the perturbation/truncation error is reached.  Note that the
    number of knots of the sparse grid are computed up to $w = 5$ for
    the pure collocation method. For the variance correction the
    sparse grid level is set to $w = 3$ since at this point the error
    is smaller than the perturbation error and there is no benefit to
    increasing $w$. The sparse grid knots needed for the variance
    correction are almost negligible compared to the pure
    collocation.}  }
\label{results:fig2}
\end{figure}

\corb{
\begin{remark}
Note that the number of knots of the sparse grid are computed equally
for the pure collocation and variance correction for this case.
However, in practice the number of sparse grid knots needed for the
variance correction are small compared to the pure collocation
approach. These is due to the fact that the variance correction is
scaled by the coefficients $\mu^f_n$ for $n = 1,\dots,N_f$.
\end{remark}
}

\corb{In Figure \ref{results:fig6}(a) and (b) the variance error decay
  plots for $k = 3$ (cubic) and $k = 4$ (quartic) are shown for the
  collocation (dashed line) and hybrid methods (solid line). The
  reference solutions are computed with a dimension adaptive sparse
  grid with 30,000 knots for the cubic case and 10,000 knots for the
  quartic case.  The collocation and hybrid estimates are computed
  with an isotropic sparse grid with Clenshaw-Curtis abscissas.}

\corb{It is observed that the error for the hybrid
  collocation-perturbation method decays faster, as the dimensions are
  increased, compared to the pure collocation method. Moreover, as the
  dimensions are increased the accuracy gain of the perturbation
  method accelerates significantly (c.f. Figure
  \ref{results:fig6}(b)). The accuracy improves from one order of
  magnitude to 23 times improvement. We expect the accuracy to further
  accelerate as we increase $w$. However, we are limited in
  computational resources to compute larger mesh sizes.}
  

\begin{figure}
\begin{center}
\begin{tabular}{cc}
\psfrag{Dimension}[bl]{\tiny Dimension}
\psfrag{Mean Error}[bl]{\tiny Mean Error}
\psfrag{Mean Error vs Dimension}[bl]{\tiny Mean Truncation Error}
\includegraphics[width=2.2in,height=2.0in]
{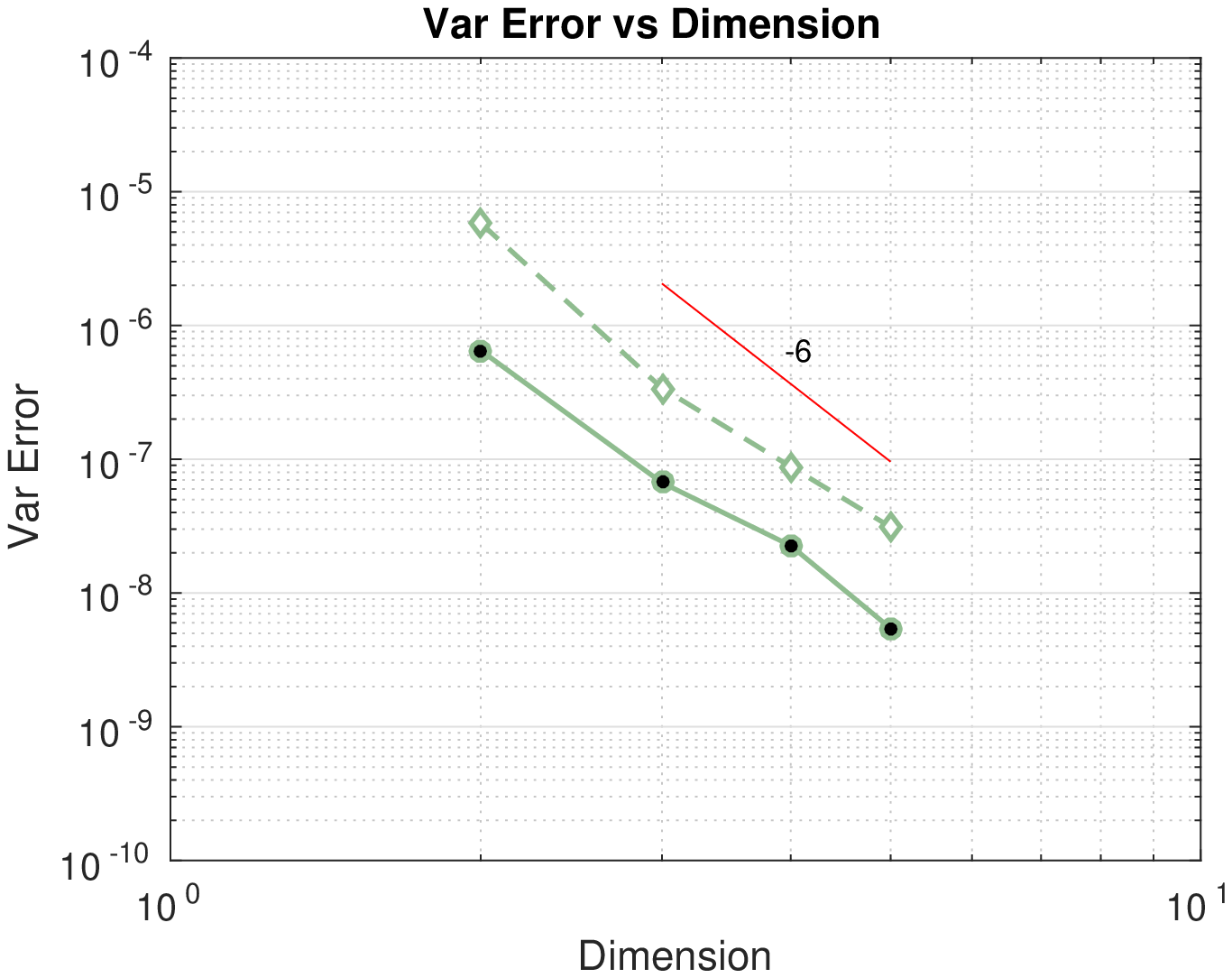} &
\psfrag{Dimension}[bl]{\tiny Dimension}
\psfrag{Var Error}[bl]{\tiny Var Error}
\psfrag{Var Error vs Dimension}[bl]{
  \tiny Variance Truncation Error}
\includegraphics[width=2.2in,height=2.0in]
{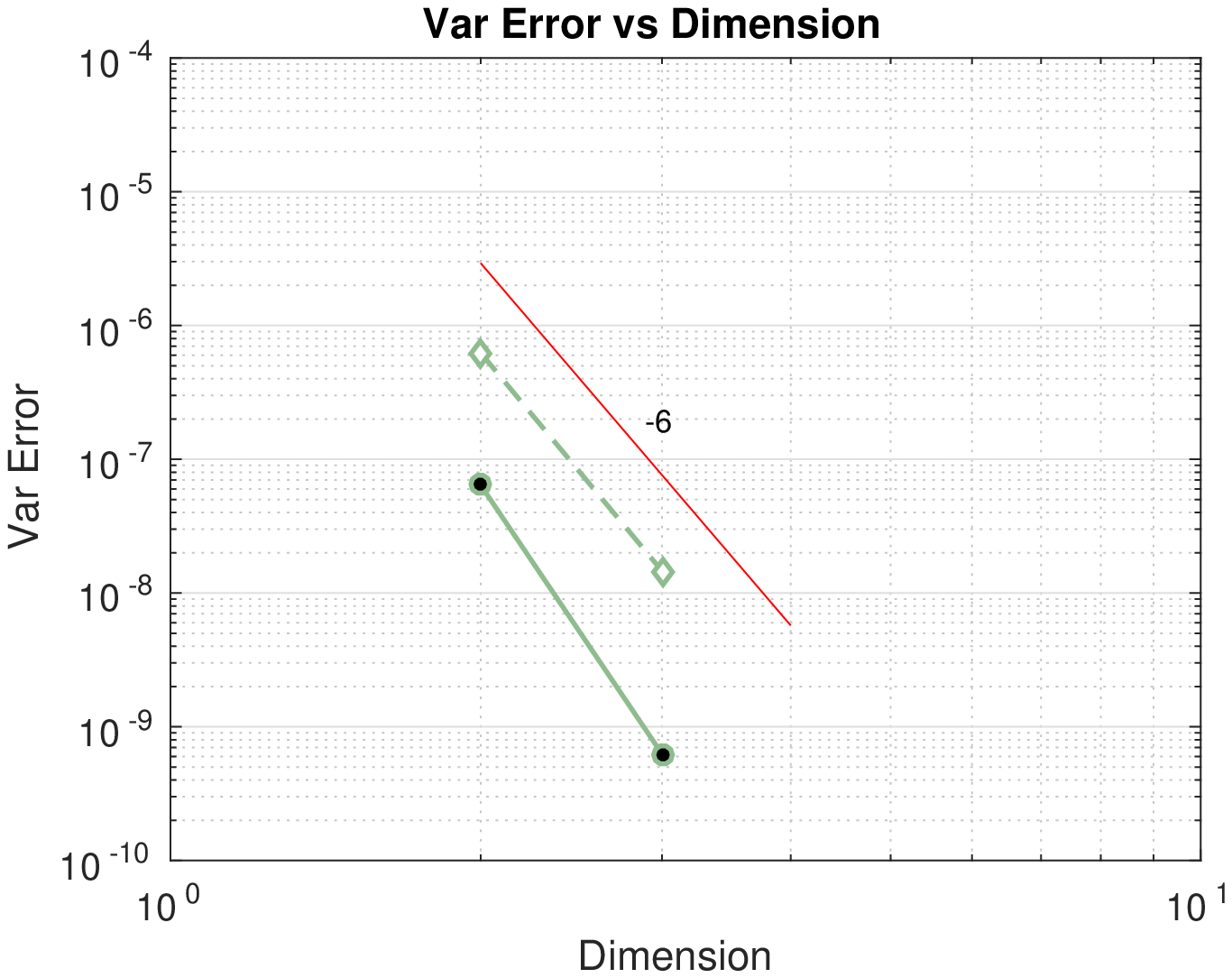} \\
(a) $c = 1/15$, $k = 3$ & (b) $c = 1/15$, $k = 4$ \\
\end{tabular}
\end{center}
\caption{Truncation error with respect to the number of dimensions and
  different decay rates.  (a) Variance error for the pure collocation
  (dashed line) and hybrid collocation-perturbation (solid line)
  methods for $c = 1/15$ and $k = 3$.  (b) Variance error ratio
  between the collocation and hybrid methods for $c = 1/15$ and $k =
  4$. Notice that the accuracy of the hybrid collocation-perturbation
  significantly increases with dimensions.}
\label{results:fig6}
\end{figure}

\section{Conclusions}

\corb{In this paper we propose a new hybrid collocation perturbation
  scheme to computing the statistics of the QoI with respect to random
  domain deformations that are split into large and small deviations.}
The large deviations are approximated with a stochastic collocation
scheme. In contrast, the small deviations components of the QoI are
approximated with a perturbation approach.

\corb{We give a rigorous convergence analysis of the hybrid approach
  based on isotropic Smolyak grids for the approximation of an
  elliptic PDE defined on a random domain.}

We show that for a linear elliptic partial differential equation with
a random domain the variance correction term can be analytically
extended to a well defined region $\Theta_{\beta,N_s}$ embedded in
$\C^{N_s}$ with respect to the random variables. This analysis leads to
a provable subexponential convergence rate of the QoI computed with an
isotropic Clenshaw-Curtis sparse grid. We show that the size of this
region, and the rate of convergence, is directly related to the decay
of the gradient of the stochastic deformation.

This approach is well suited for a moderate to a large number of
stochastic variables. Moreover we can easily extend this approach to
anisotropic sparse grids \cite{nobile2008b} to further increase the
efficiency of our approach with respect to the number of dimensions.

\bibliographystyle{plain}
\bibliography{StochasticDomainHybrid}

\end{document}